\documentclass[11pt]{article}
\usepackage{amssymb,amsmath}
\usepackage{amsthm}
\newtheorem{theorem}{Theorem}[section]
\newtheorem{theorem*}{Theorem A\!\!}
\newtheorem{proposition}{Proposition}[section]
\newtheorem{corollary}{Corollary}[section]
\newtheorem{proposition*}{Proposition A\!\!}
\newtheorem{corollary*}{Corollary A\!\!}
\newtheorem{lemma}{Lemma}[section]

\DeclareMathOperator{\Ad}{Ad}

\DeclareMathOperator{\proj}{proj}

\begin{document}

\title{ Conformally invariant trilinear forms on the sphere}

\author{Jean-Louis Clerc and Bent \O rsted}
\date{January 4, 2010}
\maketitle
\begin{abstract}
 To each complex number $\lambda$ is associated a representation $\pi_\lambda$ of the conformal group $SO_0(1,n)$ on $\mathcal C^\infty(S^{n-1})$ (spherical principal series). For three values $\lambda_1,\lambda_2,\lambda_3$, we construct a trilinear form on $\mathcal C^\infty(S^{n-1})\times\mathcal C^\infty(S^{n-1})\times \mathcal C^\infty(S^{n-1})$, which is invariant by $\pi_{\lambda_1}\otimes \pi_{\lambda_2}\otimes \pi_{\lambda_3}$. The trilinear form, first defined for $(\lambda_1, \lambda_2,\lambda_3)$ in an open set of $\mathbb C^3$ is extended meromorphically, with simple poles located in an explicit family of hyperplanes. For generic values of the parameters, we prove uniqueness of trilinear invariant forms. 
\end{abstract}

\footnotemark[0]{2000 Mathematics Subject Classification : 22E45, 43A85}

\section* {Introduction}

The motivation for this article came from the paper \cite {br}  by J. Bernstein and A. Reznikov. In order to estimate automorphic coefficients, they use trilinear invariant forms for $G=PGL_2(\mathbb R)$. Their paper shows the importance of three related questions of harmonic analysis on $PGL_2(\mathbb R)$ :

given three representations $(\pi_1,\mathcal H_1),(\pi_2, \mathcal H_2), (\pi_3, \mathcal H_3)$ in the principal spherical\footnote{with respect to the maximal compact subgroup $K=PO(2)$ of $G$} series of the group $G$,

$i)$ construct a trilinear invariant functional  on $\mathcal H_1\times \mathcal H_2\times \mathcal H_3$

$ii)$ prove uniqueness (up to a scalar) of such a trilinear invariant functional

$iii)$ compute the value of the trilinear functional on the $K$ fixed vectors of $\mathcal H_1,\mathcal H_2,\mathcal H_3$ respectively.

The representations are realized on the unit circle, on which the group $G$ acts projectively, and indexed by a complex number. One possible generalization consists in replacing the unit circle by the $n-1$ dimensional sphere $S$, under the action of the conformal group $G=SO_0(1,n)$. For this case,  we present here a construction of an invariant trilinear form (item $i)$), which uses an analytic continuation over three complex parameters and discuss the uniqueness statement (item $ii)$) for generic values of the parameters. The computation of the normalization factor (item $iii)$), even for more geometric situations, will be published elsewhere,  (cf \cite {ckop}). 

In section 1, we recall elementary facts about conformal geometry of the sphere (in particular we give a description of the orbits of $G$ in $S\times S\times S$), and introduce the representations which the paper is concerned with. In section 2, we define formally the trilinear invariant form, study the domain of convergence of the corresponding integral and determine the analytic continuation in the three complex parameters corresponding to the three representations. In section 3, we prove the uniqueness statement. The proof relies on Bruhat's theory, which we recall in an appendix at the end of the paper. In section 4, we give an alternative approach to the construction of an invariant trilinear form, using a realization of the tensor product of two of the representations involved, thus making connection with \cite{mol}.

The present paper only deals with regular values of the parameters. The residues at poles will yield new conformally invariant trilinear forms, supported by the singular orbits of $G$ in $S\times S\times S$ and involving differential operators akin the Yamabe operator on the sphere, worth of a further study. Other geometric situations are potential domains for similar results. Let $P_1,P_2,P_3$ be three parabolic subgroups of a semi-simple Lie group, such that $G$ acts on $G/P_1\times G/P_2\times G/P_3$ with a finite number of orbits (see \cite{li}, \cite{mwz}).  Take three representations induced by characters of $P_1,P_2,P_3$. Invariant trilinear forms for these three representations can plausibly be studied along the same lines as in the present paper. The case of three copies of the Shilov boundary $S$ of a bounded symmetric domain of tube type is specially appealing (see \cite{cn} for a description of the orbits of $G$ in $S\times S\times S$).

Let us mention the paper by A. Deitmar  \cite{d}, which has some overlap with our results. Previous work on the subject also includes \cite {o} and \cite{lo}.

The first author  thanks D. Barlet and L. B\'erard Bergery for conversations on various aspects of this paper and the Mathematics Department of Aarhus University for welcoming him during the preparation of the present work.

\section{Conformal geometry of the sphere}

Let \[S=\mathbb S^{n-1} = \{x=(x_1,x_2,\dots,x_n), \quad \vert x\vert^2 = x_1^2+x_2^2+\dots +x_n^2=1\} \]
be the unit sphere in $\mathbb R^n$. We usually (and tacitly) assume $n\geq 3$, as the case $n=2$ needs a few minor changes, which are occasionally mentioned in the text. The group $K=SO(n)$ operates on $S$. Let 
\[
{\bf 1^+} = (1,0,\dots,0),\quad {\bf 1^-}=(-1,0,\dots,0). 
\]
The stabilizer of ${\bf 1^+}$ in $K$ is the subgroup
\[M\simeq SO(n-1) = \Bigg\{ \begin{pmatrix} 1&0\\0&u\end{pmatrix}, u \in SO(n-1)\Bigg\}\ .
\]
With this notation, $S \simeq K/M$ is a compact Riemannian symmetric space. 

Another realization of the sphere is useful. Let $\mathbb R^{1,n}$ be the real vector space $\mathbb R^{n+1}$ with the quadratic form 
\begin{equation}\label{q}
q(x) = [x,x] = x_0^2-(x_1^2+x_2^2+\dots+x_n^2)\ .
\end{equation}
To  $x=(x_1,x_2,\dots, x_n)$ in $S$ associate $\widetilde x =(1,x_1,x_2,\dots, x_n)$ in $\mathbb R^{1,n}$. The  correspondance
\[x \quad\longmapsto \quad \mathbb R \widetilde x\]
associates to a point in $S$ an \emph{isotropic line} in $\mathbb R^{1,n}$. The correspondance is easily seen to be bijective. The group $G=SO_0(1,n)$ acts naturally on the set of isotropic lines, and hence on $S$. Explicitly, for $x$ in $S$ and $g$ in $G$, $g(x)$ is the unique point in $S$ such that 
\[\widetilde{g(x)} = {(g \widetilde x)_0}^{-1}\ g\, \widetilde x\ .
\]
For $x$ in $S$ and $g$ in $G$, set
\begin{equation} \kappa(g,x) = {(g \widetilde x)_0}^{-1}\ .\end{equation}
Let $x,y$ be in $S$. The following identity holds
\begin{equation}
[\widetilde x,\widetilde y] = 1-\langle x,y\rangle = \frac{1}{2} \vert x-y\vert^2
\end{equation}
so that for $g$ in $G$, 
\begin{equation}\label{cov}
\vert g(x)-g(y)\vert = \kappa(g,x)^{\frac{1}{2}}\ \vert x-y\vert \ \kappa(g,y)^{\frac{1}{2}}\ .
\end{equation}
The infinitesimal version of \eqref{cov} is 
\begin{equation}\label{infinv}
\vert Dg(x)\,\xi\vert =\kappa(g,x)\vert \xi\vert
\end{equation}
for $\xi$ any tangent vector to $S$ at $x$, where $g$ is in $G$ and $Dg(x)$ stands for the differential at $x$ of the map $x\mapsto g(x)$. Hence the action of $G$ on $S$ is \emph{conformal}, and $\kappa(g,x)$ is interpreted as the \emph{conformal factor} of $g$ at $x$. 
 
We look at $K$ as a subgroup of $G$. It is a maximal compact subgroup of $G$. The stabilizer in $G$ of the point $\bf 1^+$ is the parabolic subgroup $P=MAN$, where 
\[A = \left\{ a_t = \begin{pmatrix}\cosh t&\sinh t&0&\dots&0\\\sinh t& \cosh t&0&\dots&0\\0&0&1 & &\\ \vdots&\vdots&&\ddots&\\0&0&&&1\end{pmatrix},\quad t\in \mathbb R\right\}
\]
and
\[N= \left\{\ n_\xi =\begin{pmatrix}1+\frac{\vert\xi\vert^2}{2}&-\frac{\vert\xi\vert^2}{2}&&\xi^t&\\\frac{\vert \xi\vert^ 2}{2}&1-\frac{\vert \xi\vert^2}{2}&&\xi^t&\\
&&1&&
\\\xi&-\xi&&\ddots&\\&&&&1
\end{pmatrix},\quad \xi\in \mathbb R^{n-1}\right\}\ .
\]
The element $a_t$ ($t\in \mathbb R$) acts on $S$ by
\[ a_t \,\begin{pmatrix}x_1\\x_2\\\dots\\x_n\end{pmatrix} = \begin{pmatrix} \frac{\sinh t +x_1 \cosh t}{\cosh t + x_1\sinh t }\\ \frac{x_2}{\cosh t + x_1\sinh t} \\ \vdots\\ \frac{x_n}{\cosh t + x_1\sinh t}\end{pmatrix}\ .
 \]
 
 Let $\overline N$ be the subgroup image of $N$ by the standard Cartan involution of $G$ ($g\longmapsto {(g^t)}^{-1}$) :
 \[ \overline N =  \left\{\ \overline n_\xi =\begin{pmatrix}1+\frac{\vert\xi\vert^2}{2}&\frac{\vert\xi\vert^2}{2}&&\xi^t&\\-\frac{\vert \xi\vert^ 2}{2}&1-\frac{\vert \xi\vert^2}{2}&&-\xi^t&\\
&&1&&
\\\xi&\xi&&\ddots&\\&&&&1
\end{pmatrix},\quad \xi\in \mathbb R^{n-1}\right\}\ .
 \]
 The map 
 \[c : \xi \longmapsto \overline n_\xi({\bf 1^+})= \begin{pmatrix}\frac{1-\vert \xi\vert^2}{1+\vert \xi\vert^2}\\ \\ \frac{2}{1+\vert\xi\vert^2}\,\xi\\ \\ \end{pmatrix}\ .
 \] is a diffeomorphism from $ \mathbb R^{n-1}$ onto  $S\setminus\{{\bf 1^-}\} $. Its inverse is  the classical  \emph{stereographic projection} from the source ${\bf 1^-}$ onto the tangent space $T_{\bf 1^+}S$ to $S$ at ${\bf 1^+}$. When using this chart on $S$, we refer to the \emph{noncompact picture}.

\begin{proposition} The conformal factor $\kappa(g,x)$ is a smooth function of both $g$ and $x$, which satisfies the following properties : 
\medskip

$i)$ $\forall g_1,g_2\in G, x\in S$,\quad 
\begin{equation}\label{kappa}
 \kappa(g_1g_2,x) = \kappa(g_1,g_2(x))\,\kappa(g_2,x)
\end{equation}

$ii)$ $\forall g\in G, x\in S\quad  \kappa(g,g^{-1}(x)) = \kappa(g^{-1},x)^{-1}$

$iii)$  $\forall x\in S, k\in K\quad \kappa(k,x)=1$

$iv)$ $\forall x\in S, t\in \mathbb R,\quad  \kappa(a_t,x) = (\cosh t +x_1\sinh t)^{-1}$.

\end{proposition}

Let $g$ in $G$. As the dimension of the tangent space $T_xS$ is $n-1$, the Jacobian of  $g$  at $x$ is given by
\begin{equation}\label{jac}
j(g,x) = \kappa(g,x)^{n-1}\ .
\end{equation}

The map $c:\mathbb R^{n-1} \longrightarrow S\setminus \{-{\bf 1}\}$ is also conformal. In fact, one has the following relation, valid for any $\xi, \eta \in \mathbb R^{n-1}$ :

\begin{equation}
\vert c(\xi)-c(\eta)\vert = \frac{2\vert \xi-\eta\vert}{(1+\vert \xi\vert^2)^{\frac{1}{2}}(1+\vert \eta\vert^2)^{\frac{1}{2}}}\ ,
\end{equation}
and its infinitesimal version

\begin{equation}
 \vert dc(\xi)\,\zeta\vert = \frac{2}{1+\vert\xi\vert^2}\,\vert \zeta\vert
\end{equation} 
($\zeta \in \mathbb R^{n-1}$). The  corresponding integration formula reads
\begin{equation}\label{intstereo}
\int_S f(x)d\sigma(x) = \int_{\mathbb R^{n-1}} f(c(\xi)) \frac{2^{n-1}}{(1+\vert\xi\vert^2)^{n-1}}\, d\xi \ .
\end{equation}

Later, we will need a description of the orbits of $G$ in $S\times S\times S$ (for the diagonal action of $G$).  Recall first that  the group $G$ in its diagonal action on $S\times S$ has two orbits : 

\[ S^2_\top = \{(x,y)\in S\times S, x\neq y\}, \qquad \Delta_S = \{(x,x), x\in S\}.\] 
 
 As base-point in $S^2_\top$, choose $({\bf 1^+}, {\bf 1^-})$. The stabilizer of $({\bf 1^+}, {\bf 1^-})$ in $G$ is the subgroup $MA$.

\begin{proposition}\label{orbit} Let $n\geq 3$. There are 5 orbits of $G$ in $S\times S\times S$, namely
\begin{equation*}
\begin{split}
{\mathcal O}_0 &= \{ (x_1,x_2,x_3), x_i\neq x_j\ \rm{for\ } i\neq j\}\\
{\mathcal O}_1 &= \{(x_1,x,x), x\neq x_1\}\\
{\mathcal O}_2 &= \{(x,x_2,x), x\neq x_2\}\\
{\mathcal O}_3 &= \{(x,x,x_3), x\neq x_3\}\\
{\mathcal O}_4 &= \{(x,x,x)\}\quad .
\end{split}
\end{equation*}
\end{proposition}
\begin{proof} The five subsets of $S\times S\times S$ are invariant under the diagonal action of $G$. So it suffices to show that $G$ is transitive on each of these sets. 

$\bullet$ $G$ is transitive on $S$, hence on $\mathcal O_4$. Choose $({\bf 1^+}, {\bf 1^+}, {\bf 1^+})$ as base-point. The stabilizer of the base-point in $G$ is the subgroup $P=MAN$.

$\bullet$ The stabilizer $P$ of $\bf 1^+$ is transitive on $S\setminus \{\bf {1^+}\}$ (the action of $N$ is already transitive on $S\setminus \{\bf {1^+}\}$), so that $G$ is transitive on $\mathcal  O_j$, for $j=1,2,3$. As base-point in $\mathcal O_3$ (similar choices can be made for $\mathcal O_1$ and $\mathcal O_2$) choose $({\bf 1^+}, {\bf 1^+}, {\bf 1^-})$. The stabilizer of $({\bf 1^+}, {\bf 1^+}, {\bf 1^-})$ is the subgroup $MA$.

$\bullet $ Let $x_1,x_2,x_3$ be in $\mathcal O_0$. We may assume w.l.o.g. that $x_1 = {\bf 1^+}, x_2 = {\bf 1^-}$ and $x_3\neq \mathbf 1^{\pm}$. The stabilizer of $(\mathbf 1^+, \mathbf 1^-)$ in $G$ is $M A$. The orthogonal projection of $x_3$ on the hyperplane orthogonal to $(\bf 1^+, \bf 1^-)$ is not $0$ and can be mapped by $M$ to $c e_2$, with $0<c<1$, so that there exists $t\in \mathbb R$ such that $x_3$ is conjugate under $M$ to the point $\displaystyle (\tanh t, \frac{1}{\cosh t}, 0, \dots, 0)= a_t(e_2)$. Hence any triplet in $\mathcal O_0$ is conjugate under $G$ to the triplet $({\bf 1}^+, {\bf 1}^-, e_2)$. Its stabilizer  in $G$ is the compact subgroup 
$ \{k\in M, ke_2 = e_2\}\simeq SO(n-2)$.
\end{proof}
When $n=2$ ($S$ is the unit circle), then there are \emph{two} open orbits in $S\times S\times S$ unde the action of $SO_0(1,2)$, each characterized by the value of the \emph{orientation index} of the three points in $S$. It is possible to remedy to this fact by using the slightly larger  (no longer connected) group $O(1,2)$ instead of $SO_0(1,2)$.

Let $d\sigma$ be the Lebesgue measure on $ S$ and let $\omega_{n-1}=\int_ S d\sigma(x)$. Also set $\rho=\frac{n-1}{2}$. Under the action of $G$, the measure is $d\sigma$ is transformed according to
\begin{equation}\label{intS}
\int_ S f\big(g^{-1}(x)\big) d\sigma(x) = \int_ S f(y) \kappa (g,y)^{2\rho}d\sigma(y)
\end{equation}

Let $\lambda$ be in $\mathbb C$. For $f$ in ${\mathcal C}^\infty( S)$, the formula
\begin{equation}
\pi_\lambda (g) f (x) = \kappa(g^{-1},x)^{\rho+\lambda} \ f(g^{-1}(x))
\end{equation}
defines a \emph{representation} of the group $G$, which will be denoted by $\pi_\lambda$.  It is a continous representation when the space 
$\mathcal C^\infty (S)$ is equipped with its natural Fr\'echet topology (see \cite{t} for a systematic study of these representations).

The representations $\pi_\lambda$ and $\pi_{-\lambda}$ are dual  in the sense that, for all $\varphi, \psi\in \mathcal C^\infty
(S)$
\begin{equation}\label{dua}
\int_S \pi_{-\lambda}(g) \varphi (s) \pi_\lambda(g)\psi(s) ds = \int_S \varphi(s) \psi(s) ds \,
\end{equation}
as can be deduced from the change of variable formula \eqref{intS}. For $\lambda$ \emph{pure imaginary}, the representation $\pi_\lambda$ can be extended continuously to $L^2(S)$ to yield a \emph{unitary} representation of $G$ (this is the reason for using $\rho+\lambda$ in the definition of $\pi_\lambda$). Observe that the action of $K$ is independant of $\lambda$ and the constant fonction $\mathbb I_S$ (equal to $1$ evereywhere) is fixed by the action of $K$.

For $\alpha$ in $\mathbb C$, let $k_\alpha$ be the kernel on $ S\times S$ defined by \[k_\alpha(x,y) = \vert x-y\vert^{-\rho +\alpha}.\]
 It satisfies the following \emph{transformation property} under the action of an element $g$ of $G$:
\begin{equation}\label{lawk}
k_\alpha (g(x),g(y)) = \kappa(g,x)^{-\frac{\rho}{2}+\frac{\alpha}{2}}\ k_\alpha(x,y)\ \kappa(g,y)^{-\frac{\rho}{2}+\frac{\alpha}{2}}\ 
\end{equation}
for all $x,y$ in $ S$.

\section{Construction of an invariant trilinear form}

\subsection{Formal construction}

Let $\lambda_1,\lambda_2,\lambda_3$ be three complex numbers. Let $\mathcal T$ be a continuous trilinear form from $\mathcal C^\infty( S)\times \mathcal C^\infty( S)\times \mathcal C^\infty( S)$ into $\mathbb C$. The functional $\mathcal T$ is said to be \emph{invariant} w.r.t. $ \pi_{\lambda_1},\pi_{\lambda_2},\pi_{\lambda_3}$ if, for every $f_1,f_2,f_3$ in $\mathcal C^\infty( S)$, and every $g$ in $G$,
\begin{equation}
\mathcal T(\pi_{\lambda_1} (g)f_1,\pi_{\lambda_2}(g)f_2, \pi_{\lambda_3}(g)f_3)= \mathcal T(f_1,f_2,f_3)\ .
\end{equation}

By Schwartz's \emph{kernel theorem}, there exists a unique distribution $T$ on $S\times S\times S$, such that
\begin{equation}\mathcal T(f_1,f_2,f_3) = T(f_1\otimes f_2\otimes f_3)
\end{equation}
where, as usual, $f_1\otimes f_2\otimes f_3$ is the function on $S\times S\times S$ defined by \[(f_1\otimes f_2\otimes f_3) (x_1,x_2,x_3) =f_1(x_1)f_2(x_2)f_3(x_3).\]

Let $\alpha_1,\alpha_2,\alpha_3$ be complex numbers, and set $\boldsymbol \alpha = (\alpha_1,\alpha_2,\alpha_3)$. Let $K_{\boldsymbol \alpha}$ be the kernel on $S\times S\times S$ defined by
\begin{equation}\label{kerK}
K_{\boldsymbol \alpha}(x_1,x_2,x_3) = k_{\alpha_1}(x_2,x_3)\,  k_{\alpha_2}(x_3,x_1)\,k_{\alpha_3}(x_1,x_2)
\end{equation}

For $f_1,f_2,f_3$ three functions in ${\mathcal C}^\infty( S)$, define the trilinear functional $\mathcal K_{\boldsymbol \alpha}$ by
\begin{equation}\label{distK}
\mathcal K_{\boldsymbol \alpha}(f_1,f_2,f_3) =
\int_ {S\times S\times S}\!\!\! \!\!\!\!\!\!K_{\boldsymbol \alpha}(x_1,x_2,x_3)f_1(x_1)f_2(x_2)f_3(x_3) d\sigma(x_1)\, d\sigma(x_2)\, d\sigma(x_3)\
\end{equation}
whenever it makes sense. 
\begin{theorem}\label{Kalpha} Let $\boldsymbol \lambda = (\lambda_1,\lambda_2,\lambda_3)\in \mathbb C^3$. Define $\boldsymbol \alpha =(\alpha_1,\alpha_2,\alpha_3)$ by the relations
\begin{equation}\label{param}
\begin{split}
\alpha_1 &= -\lambda_1+\lambda_2+\lambda_3\\
\alpha_2 &= -\lambda_2+\lambda_3+\lambda_1\\
\alpha_3 &= -\lambda_3+\lambda_1+\lambda_2 \ .
\end{split}
\end{equation}
Then
\begin{equation}
\mathcal K_{\boldsymbol \alpha} (\pi_{\lambda_1} (g)f_1,\pi_{\lambda_2}(g)f_2, \pi_{\lambda_3}(g)f_3)=\mathcal K_{\boldsymbol \alpha}(f_1,f_2,f_3)\ ,
\end{equation}
whenever the integral on the right handside is defined.
\end{theorem}

The proof is obtained by  the change of variables $y_j = g^{-1}(x_j)$ ($j=1,2,3$) in the integral defining the left hand-side, using \eqref{lawk} and  \eqref{intS}. Observe that the right-handside integral converges if, for $i\neq j$,  the supports of $f_i$ and $f_j$ are disjoint, or if $\int_{S\times S\times S} \vert K_{\boldsymbol \alpha}(x_1,x_2,x_3)\vert\, d\sigma(x_1)d\sigma(x_2)d\sigma(x_3)< +\infty$.
\medskip

The version of the trilinear functional in the noncompact picture will be useful. Let $\boldsymbol \alpha=(\alpha_1,\alpha_2,\alpha_3) \in \mathbb C^3$ and set for $y_1,y_2,y_3 \in \mathbb R^{n-1}$
\begin{equation}\label{kerJ}
 J_{\boldsymbol \alpha} (y_1,y_2,y_3) = \vert y_1-y_2\vert^{-\rho+\alpha_3}  \vert y_2-y_3\vert^{-\rho+\alpha_1} \vert y_3-y_1\vert^{-\rho+\alpha_2}\ 
\end{equation}
and, whenever it makes sense, let $\mathcal J_{\boldsymbol \alpha}$ be the associated distribution on $\mathbb R^{n-1}\times \mathbb R^{n-1}\times \mathbb R^{n-1}$ given by
\begin{equation}\label{distJ}
\mathcal J_{\boldsymbol \alpha}(\varphi) = \int J_{\boldsymbol \alpha} (y_1,y_2,y_3) \,\varphi(y_1,y_2,y_3)\, dy_1\, dy_2\,dy_3\ ,
\end{equation}
$(\varphi\in \mathcal C^\infty(\mathbb R^{n-1}\times \mathbb R^{n-1}\times \mathbb R^{n-1})$). Moreover, let  $\Psi_{\boldsymbol \alpha}$ be the function defined on $\mathbb R^{n-1}\times \mathbb R^{n-1}\times \mathbb R^{n-1}$ by
\[\Psi_{\boldsymbol \alpha}(y_1,y_2,y_3) = (1+\vert y_1\vert^2)^{-\rho-\frac{\alpha_2}{2}-\frac{\alpha_3}{2}}(1+\vert y_2\vert^2)^{-\rho-\frac{\alpha_3}{2}-\frac{\alpha_1}{2}}(1+\vert y_3\vert^2)^{-\rho-\frac{\alpha_1}{2}-\frac{\alpha_2}{2}} \ .\]

\begin{proposition} Let $f \in \mathcal C^\infty(S\times S\times S)$. Then
\begin{equation}\label{ncint}
\mathcal K_{\boldsymbol \alpha} (f) =2^{3(n-1)} \mathcal J_{\boldsymbol \alpha}\big( (f\circ c) \Psi_{\boldsymbol \alpha}\big)
\end{equation}
whenever the left hand side is defined.
\end{proposition}
This is merely the change of variable  \eqref{intstereo} in the integral \eqref{distK}. 
\subsection{ Integrability of the kernel $K_{\boldsymbol \alpha}$}

\begin{theorem}\label{abs} The kernel $K_{\boldsymbol\alpha}$ is integrable on $S\times S\times S$ if and only if
\begin{equation}\label{conv123}
\Re \alpha_j >-\rho,\quad j=1,2,3
\end{equation}
\begin{equation}\label{conv4}
\Re \alpha_1+\Re \alpha_2 +\Re \alpha_3 >-\rho
\end{equation}
\end{theorem}

\begin{proof} It is enough to study the integral when the $\alpha$'s are real, in which case the kernel $K_{\boldsymbol \alpha}$ is positive. Let $\mathcal U$ be a (small) neighborhood of $({\bf 1}^+, {\bf 1}^+, {\bf 1}^+)$ in $S\times S\times S$. Let $g$ be in $G$. As $x$ varies in $S$, $j(g,x)$ remains bounded from below and from above. Thanks to the transformation property of the kernel $k_\alpha$  \eqref {lawk}, the integrals  of $K_{\boldsymbol \alpha}$ over $\mathcal U$ and over $g(\mathcal U)$ are of the same nature. As $\mathcal U$ meets all $G$-orbits, a partition of unity argument shows that the integrability over $S$ of the kernel $K_{\boldsymbol\alpha}$ is equivalent to its integrability over $\mathcal U$. Now use the stereographic projection to see that the integrability of $K_{\boldsymbol \alpha}$ over $\mathcal U$ is equivalent to the integrability of $J_{\boldsymbol \alpha}$ over $c^{-1}(\mathcal U)$, which is a (small) neighborhood of $(0,0,0)$ in $\mathbb R^{n-1}\times \mathbb R^{n-1}\times \mathbb R^{n-1}$.

Hence, we are reduced to study the convergence of the integral
\[ I_1=\int_{\substack{\vert \xi_1\vert<\delta\\\vert \xi_2\vert<\delta\\\vert \xi_3\vert<\delta}} \vert \xi_1-\xi_2\vert^{\alpha_3-\rho} \vert \xi_2-\xi_3\vert ^{\alpha_1-\rho}\vert \xi_3-\xi_1\vert^{\alpha_2-\rho} \,d\xi_1\,d\xi_2\,d\xi_3\ ,
\]
where $\delta$ is a small positive number.
Set
\[ y_1=\xi_1, \quad y_2=\xi_1-\xi_2,\quad y_3 = -\xi_2+\xi_3\ .\]
Then, after integrating with respect to $y_1$, the integral $I_1$ is seen to be of the same nature as the integral
\[I_2=\int_{\substack{\vert y_2\vert <\delta\\ \vert y_3\vert <\delta}}
\vert y_2\vert^{\alpha_3-\rho}\vert y_3\vert^{\alpha_2-\rho}\vert y_2-y_3\vert^{\alpha_1-\rho}\, dy_2\, dy_3\ .
\]

Let $\Sigma$ be the unit ball in $\mathbb R^{(n-1)}\times \mathbb R^{(n-1)}$, and let $d\boldsymbol\sigma$ be the Lebesgue measure on $\Sigma$. Use polar coordinates $(r,(\sigma_2,\sigma_3))$ defined by
\[y_2= r\sigma_2,\ y_3= r\sigma_3,\quad r^2 = \vert y_2\vert^2+\vert y_3\vert^2,\quad \vert\sigma_1\vert^2+\vert\sigma_3\vert^2=1\ 
\] to obtain that $I_2$ is of the same nature as
\[ I_3= \int_0^{\delta} r^{\alpha_1+\alpha_2+\alpha_3-3\rho+2(n-1)-1} dr \times \int_\Sigma \vert \sigma_2\vert^{\alpha_3-\rho}\vert \sigma_3\vert^{\alpha_2-\rho}\vert \sigma_2-\sigma_3\vert^{\alpha_1-\rho}d\boldsymbol\sigma \ .
\]
The first factor converges if and only if condition \eqref{conv4} is satisfied. It remains to discuss the convergence of 
\[I_4 = \int_\Sigma \vert \sigma_2\vert^{\alpha_3-\rho}\vert \sigma_3\vert^{\alpha_2-\rho}\vert \sigma_2-\sigma_3\vert^{\alpha_1-\rho}d\boldsymbol\sigma \ .
\]
Let $\delta>0$ and  consider the following  open subsets of $\Sigma$ :
\[ \Sigma_2 = \{ (\sigma_2,\sigma_3), \vert \sigma_2\vert<\delta\},\  \Sigma_3 = \{ (\sigma_2,\sigma_3), \vert \sigma_3\vert<\delta\}, \ \Sigma_1 =\{(\sigma_2,\sigma_3), \vert \sigma_2-\sigma_3\vert<\delta\}\ .
\]
For $\delta$ small enough, these sets are disjoint (recall that $\vert \sigma_2\vert^2+\vert \sigma_3\vert^2 = 1$). On $ \Sigma\setminus (\Sigma_1\cup\Sigma_2\cup\Sigma_3)$ the function to be integrated is bounded from below. Hence
the integral $I_4$ is convergent if and only the integrals
\[K_j=\int_{\Sigma_j} \vert \sigma_2\vert^{\alpha_3-\rho}\vert \sigma_3\vert^{\alpha_2-\rho}\vert \sigma_2-\sigma_3\vert^{\alpha_1-\rho}d\boldsymbol \sigma \]
 are convergent for $j=1,2,3$. Let $j=2$. On $\Sigma_2$, both $\vert \sigma_3\vert$ and $\vert \sigma_2-\sigma_3\vert$ are bounded from below, so that it is equivalent to study the convergence of the integral 
 \[\int_{\Sigma_2} \vert \sigma_2\vert^{\alpha_3-\rho} d\boldsymbol \sigma\ .\]
 We are reduced to a classical situation and $K_2$ converges if and only if $ \alpha_3-\rho > -(n-1)$. A similar study applies to $K_1$ and $K_3$. This completes the proof of Theorem \ref{abs}.
\end{proof} 

\begin{corollary}
Let $\lambda_1,\lambda_2,\lambda_3$ be three complex numbers satisfying the condition
\[ 0\leq \Re (\lambda_j)<\rho\ .
\]
Define $\boldsymbol \alpha$ by the relations \eqref{param}. Then the kernel $K_{\boldsymbol \alpha} $ is integrable, and the corresponding trilinear form is invariant for $\pi_{\lambda_1}, \pi_{\lambda_2}, \pi_{\lambda_3}$.
\end{corollary}
\begin{proof}
The conditions on $\boldsymbol \lambda= (\lambda_1,\lambda_2,\lambda_3)$ imply that $\Re(\alpha_j)>-\rho $ for $j=1,2,3$ and $\Re( \alpha_1+\alpha_2+\alpha_3)\geq 0$. Hence $\boldsymbol \alpha$ is in the domain of integrability of $K_{\boldsymbol \alpha}$.
\end{proof}
The corollary covers all interesting cases for spherical unitary series, provided one excludes the trivial representation. In fact, the parameter $\lambda$  for such a representation can be chosen either as pure imaginary (principal series) or satisfying $0<\lambda<\rho$ (complementary series, excluding the trivial representation).

\subsection{Analytic continuation of $\mathcal K_{\boldsymbol \alpha}$}
The main result of this section concerns the analytic continuation of $\mathcal K_{\boldsymbol \alpha}$ beyond its domain of integrability.

\begin{theorem}\label{mero}
 The map $\boldsymbol \alpha \longmapsto \mathcal K_{\boldsymbol \alpha}$, originally defined for $\boldsymbol \alpha$ in  in the domain described by the conditions  \eqref{conv123} and \eqref{conv4}, valued in $\mathcal D'(S\times S\times S)$ can be extended  meromorphically with at most simple poles  along the family of hyperplanes in $\mathbb C\times \mathbb C\times\mathbb C$  defined by the equations
\begin{equation}\label{poles1}
 \alpha_j = -\rho-2k,\quad j=1,2,3,\quad k\in \mathbb N
 \end{equation}
 \begin{equation}\label{poles2}
 \alpha_1+\alpha_2+\alpha_3 = -\rho-2k,\quad k\in \mathbb N
\end{equation}
\end{theorem}
\begin{proof} Let $f$ be in $\mathcal C^\infty(S\times S\times S)$ and consider the integral 
\[\int_{S\times S\times S} K_{\boldsymbol \alpha} (x_1,x_2,x_3)f(x_1,x_2,x_3) d\sigma(x_1)d\sigma(x_2)d\sigma(x_3)\ .
\]
to be meromorphically continued. Repeating the argument given supra during the discussion of the integrability  of the kernel $K_{\boldsymbol \alpha}$, we may assume that $f$ has its support contained in a small neigborhood of the point $(\bf 1^+,\bf 1^+,\bf 1^+)$. Further, transfer the integral to the noncompact picture (cf \eqref{ncint}), and study the analytic continuation of 
\[ \boldsymbol \alpha \longmapsto \mathcal J_{\boldsymbol\alpha}\big((f\circ c)\Psi_{\boldsymbol \alpha}\big)
\]
Now $\varphi = f\circ c$ is in $\mathcal C^\infty_c(\mathbb R^{n-1}\times \mathbb R^{n-1}\times \mathbb R^{n-1})$, and both $\boldsymbol\alpha \longmapsto \Psi_{\boldsymbol \alpha}$ and $\boldsymbol\alpha \longmapsto {\Psi_{\boldsymbol \alpha}}^{-1}$ are entire fonctions  on $\mathbb C^3$, so that it is equivalent to study the meromorphic continuation of $\mathcal J_{\boldsymbol \alpha}$ as a distribution on $\mathbb R^{n-1} \times \mathbb R^{n-1}\times \mathbb R^{n-1}$. 

The kernel $J_{\boldsymbol \alpha}$ is invariant by translations by "diagonal vectors". To take advantage of this remark, make the change of variables
\[z_1 = y_1,\quad z_2 = y_1-y_3,\quad z_3 = y_1-y_2
\] 
in the integral 
\[ \mathcal J_{\boldsymbol \alpha} (\varphi) = \int \vert y_1-y_2\vert^{-\rho+\alpha_3} \vert y_2-y_3\vert^{-\rho+\alpha_1}\vert y_3-y_1\vert^{-\rho+\alpha_2}\varphi(y_1,y_2,y_3) \,dy_1\, dy_2 \,dy_3
\]
to obtain
\[ \mathcal J_{\boldsymbol \alpha} (\varphi) = \int_{\mathbb R^{n-1}\times \mathbb R^{n-1}} \vert z_2\vert^{-\rho+\alpha_2}\vert z_3\vert^{-\rho+\alpha_3} \vert z_2-z_3\vert^{-\rho +\alpha_1}\psi(z_2,z_3) \,dz_2\, dz_3
\]
where we have set
\[\psi(z_2,z_3) = \int_{\mathbb R^{n-1}} \varphi(z_1,z_1-z_3, z_1-z_2) dz_1\ .
\]
Now observe that $\psi$ is in $\mathcal C^\infty_c(\mathbb R^{n-1}\times \mathbb R^{n-1})$.  Hence we are reduced to studying the analytic continuation of the distribution $\mathcal I_{\boldsymbol \alpha}$ on $\mathbb R^{n-1}\times \mathbb R^{n-1}$ defined by
\begin{equation}\label{distI}
\mathcal I_{\boldsymbol \alpha} (\psi) = \int_{\mathbb R^{n-1}\times \mathbb R^{n-1}} I_{\boldsymbol\alpha} (z_2,z_3)\psi(z_2,z_3) \,dz_2\, dz_3 
\end{equation}
for $\psi$ in $\mathbb C^\infty_c(\mathbb R^{n-1}\times \mathbb R^{n-1})$, where we set
\begin{equation}\label{kerI}
I_{\boldsymbol \alpha} (z_2,z_3) = \vert z_2\vert^{-\rho+\alpha_2}\vert z_3\vert^{-\rho+\alpha_3} \vert z_2-z_3\vert^{-\rho +\alpha_1}\ .
\end{equation}

For $\delta>0$, consider the following open subsets of $ \mathbb R^{n-1}\times \mathbb R^{n-1}$ 
\[\mathcal U_0 = \{ (z_2,z_3) , \vert z_2\vert<\delta, \vert z_3\vert <\delta, \vert z_2-z_3\vert <\delta\}\] 
\[\mathcal U_1 = \{ (z_2,z_3), \vert z_2\vert>\frac{\delta}{2}, \vert z_3\vert >\frac{\delta}{2}, \vert z_2-z_3\vert <\frac{\delta}{2}\}\]
\[\mathcal U_2 = \{ (z_2,z_3), \vert z_2\vert <\frac{\delta}{2}, \vert z_3\vert>\frac{\delta}{2}, \vert z_3-z_2\vert>\frac{\delta}{2}\}\]
\[\mathcal U_3 = \{ (z_2,z_3), \vert z_2\vert >\frac{\delta}{2}, \vert z_3\vert <\frac{\delta}{2}, \vert z_2-z_3\vert >\frac{\delta}{2}\}\]
\[ \mathcal  U_\infty = \{ (z_2,z_3), \vert z_2\vert >\frac{\delta}{2}, \vert z_3\vert > \frac{\delta}{2}, \vert z_2-z_3\vert >\frac{\delta}{2}\}\ .
\]
The family of these five open sets form a covering of $\mathbb R^{n-1}\times \mathbb R^{n-1}$. Let study the restriction of the distribution $\mathcal I_{\boldsymbol \alpha}$ to each of these five open subsets.

If $Supp(\psi) \subset \mathcal U_\infty$, there the $\mathcal I_{\boldsymbol \alpha}(\psi)$ extends as an entire function, because $I_{\boldsymbol \alpha}$ has no singularity on $\mathcal U_\infty$. Next assume that $Supp(\psi)\subset \mathcal U_2$. Set
\[\phi_{\boldsymbol \alpha}(z_2) = \int_{\mathbb R^{n-1}}\vert z_3\vert^{-\rho+\alpha_3} \vert z_2-z_3\vert^{-\rho+\alpha_1}\psi(z_2,z_3)dz_3\ ,
\]
so that 
\begin{equation}\label{onevar}
 \mathcal I_{\boldsymbol \alpha}(\psi) = \int_{\mathbb R^{n-1}} \vert z_2\vert^{-\rho+\alpha_2} \phi_{\boldsymbol \alpha}(z_2)dz_2\ . 
\end{equation}
As $\vert z_3\vert$ and $\vert z_2-z_3\vert$ are bounded from below on $\mathcal U_2$, the function $\phi_{\boldsymbol \alpha}$ is in $\mathcal C^\infty_c(\mathbb R^{n-1})$ and $\boldsymbol \alpha \longmapsto \phi_{\boldsymbol \alpha}$ is an entire function on $\mathbb C^3$ . On $\mathbb R^{n-1}$, the distribution-valued function  $s\longmapsto \vert z_2\vert^s$ extends meromorphically on $\mathbb C$, with simple poles at $s = -(n-1)-2k, k\in \mathbb N$ (see  \cite{gs}), so that the integral \eqref{onevar} extends meromorphically to $\mathbb C^3$ with simple poles along the hyperplanes $\alpha_2 = -\rho-2k$. A similar analysis can be done over $\mathcal U_1$ and $\mathcal U_3$. 

To sum up what we have already proved, introduce  the family $\mathcal M$ of meromorphic functions on $\mathbb C^3$ having at most simple poles along the hyperplanes $\{\alpha_j = -\rho -2k\}$, $j=1,2,3, k\in \mathbb N$. Notice that they are the hyperplanes corresponding to the conditions \eqref{poles1}.

\begin{proposition}\label{mero2}
 Let $\psi$ be in $\mathcal C^\infty_c(\mathbb R^{n-1}\times \mathbb R^{n-1})$ and assume that $Supp(\psi) \not\ni (0,0)$ . Then the function $\boldsymbol \alpha \mapsto \mathcal I_{\boldsymbol \alpha}(\psi)$ belongs to the class $\mathcal M$.
\end{proposition}

\begin{proof}
In fact, choose $\delta$ small enough so that $Supp(\psi) \cap \mathcal U_0 = \emptyset$. Use a partition of unity to write $\psi$ as \[\psi = \psi_1+\psi_2+\psi_3+\psi_\infty\] where $\psi_j\in \mathcal C^\infty_c(\mathbb R^{n-1}\times \mathbb R^{n-1})$  and 
$Supp(\psi_j)\subset \mathcal U_j, j=1,2,3,\infty$. The previous analysis shows that $\boldsymbol \alpha \longmapsto \mathcal I_{\boldsymbol \alpha}(\psi_j)$ extends meromorphically on $\mathbb C^3$ with at most simple poles along the hyperplanes $\alpha_j = -\rho-2k,\ k\in \mathbb N$ for $j=1,2,3$, whereas $\mathcal I_{\boldsymbol\alpha}(\psi_\infty)$ is an entire function of $\boldsymbol \alpha$.
\end{proof}
Now we use \emph{a priori} the existence of the meromorphic continuation of such integrals (see \cite {sab}). Moreover,  the poles are located on a locally finite family of affine hyperplanes (of a rather specific type, but we won't need this result). Let $\mathcal H$ be such an hyperplane (to be determined), but not included in the family of hyperplanes given by conditions \eqref{poles1}. Let $\boldsymbol\alpha^0=(\alpha_1^0,\alpha_2^0, \alpha_3^0)$ be a regular point in $\mathcal H$, (i.e. not contained in any other hyperplane of poles). The Laurent coefficients at $\boldsymbol\alpha^0$ are distributions on $\mathbb R^{n-1}\times \mathbb R^{n-1}$, and, by Lemma \ref{mero2}, their supports have to be contained in $\{0,0\}$, hence they are derivatives of the Dirac measure $\delta_{(0,0)}$. So, if $\mathcal I_{\boldsymbol \alpha}$ does have a pole at $\boldsymbol\alpha^0$, there exists  a smooth function $\rho$ on $\mathbb R^{n-1}\times \mathbb R^{n-1}$ with compact support and identically equal to $1$ in a neigbourhood of $(0,0)$, and  a polynomial $P$ on $\mathbb R^{n-1}\times \mathbb R^{n-1}$, homogeneous of degree $k$ such that $\mathcal I_{\boldsymbol \alpha} (\rho P)$ does not extend holomorphically at $\boldsymbol \alpha^0$. For $t\in\mathbb R^*$, let $\rho_t(z_2,z_3) = \rho(tz_2,tz_3)$. Use the change of variables $(z_2\mapsto tz_2, z_3\mapsto tz_3)$ to get \footnote{Following a traditional way, we write ${\mathcal I}_{\boldsymbol \alpha}(\varphi)$ as an integral. What is really used here is merely the homogeneity of ${\mathcal I}_{\boldsymbol \alpha}$. }  , for $\alpha\neq \alpha_0$
\[
{\mathcal I}_{\boldsymbol \alpha} (\rho_t P) = \int_{\mathbb R^{n-1}\times \mathbb R^{n-1}}\!\!\!\!\! \!\!\!\!\!\!\vert  z_2\vert^{-\rho+\alpha_3}\vert z_3 \vert^{-\rho+\alpha_2}\vert z_2-z_3\vert^{-\rho+\alpha_1}\,\rho(tz_2,tz_3) P(z_2,z_3) \,dz_2\, dz_3\ 
\]
\[ =\vert t\vert^{-\alpha_3-\alpha_2-\alpha_1-\rho}\, t^{-k}\,\mathcal I_{\boldsymbol \alpha}( \rho P)\ .\]

Hence
\begin{equation}\label{hompole}
(1-\vert t\vert^{-\alpha_3-\alpha_2-\alpha_1-\rho}t^{-k}) \mathcal I_{\boldsymbol \alpha}(\rho P) = \mathcal I_{\boldsymbol \alpha}\big( (\rho-\rho_t)P\big)\ .
\end{equation}
Now observe that the support of $(\rho-\rho_t)P$ does not contain $(0,0)$, and hence, $\boldsymbol \alpha \longmapsto \mathcal I_{\boldsymbol \alpha} \big( (\rho-\rho_t)P\big)$ belongs to $\mathcal M$. The assumption that $\mathcal I_{\boldsymbol \alpha}(\rho P)$ does have a pole at $\boldsymbol\alpha_0$ forces the condition 
\[ \forall t\in \mathbb R^*\qquad1-\vert t\vert^{-\alpha_3^0-\alpha_2^0-\alpha_1^0-\rho}t^{-k} =0\ .
\]
  
In turn, this condition amounts to
\[ k\in 2\mathbb N,\qquad \alpha_1^0+\alpha_2^0+\alpha_3^0 = -\rho -k\ .\]
Moreover, \eqref{hompole} shows that $\mathcal I_{\boldsymbol \alpha}$ has at most a simple pole along the corresponding hyperplane. This achieves the proof of Theorem \ref{mero}. 
 \end{proof}
 As the invariance condition remains true by analytic continuation, Theorem \ref{mero} can be reformulated for trilinear invariant functionals (cf Theorem  \ref{Kalpha}).
 
 \begin{theorem} Let $\boldsymbol \lambda = (\lambda_1,\lambda_2,\lambda_3)$ in $\mathbb C^3$. Assume that
\begin{equation}\label{mero3}
\begin{split}
-\lambda_1+\lambda_2+\lambda_3 \notin -\rho-2\mathbb N\\
\lambda_1-\lambda_2+\lambda_3 \notin -\rho-2\mathbb N\\
\lambda_1+\lambda_2-\lambda_3 \notin -\rho-2\mathbb N\\
\lambda_1+\lambda_2+\lambda_3 \notin -\rho-2\mathbb N\\
\end{split}
\end{equation}
Set $\boldsymbol \alpha = (\alpha_1,\alpha_2,\alpha_3)$ where (cf \eqref{param})
\begin{equation*}
\begin{split}
\alpha_1 &= -\lambda_1+\lambda_2+\lambda_3\\
\alpha_2 &= -\lambda_2+\lambda_3+\lambda_1\\
\alpha_3 &= -\lambda_3+\lambda_1+\lambda_2 \ .
\end{split}
\end{equation*}
Then $(f_1,f_2,f_3)\longmapsto \mathcal T_{\boldsymbol \lambda}(f_1,f_2,f_3) = \mathcal K_{\boldsymbol \alpha} (f_1\otimes f_2\otimes f_3)$ is a well defined non trivial trilinear invariant functional w.r.t.  the representations $(\pi_{\lambda_1}, \pi_{\lambda_2}, \pi_{\lambda_3})$.
\end{theorem}

The next result was obtained some time ago by the present authors, and  has been generalized to other geometric situations in a collaboration with T. Kobayashi and M. Pevzner (see \cite{ckop}). To state the result, consider the evaluation of the functional $\mathcal K_{\boldsymbol \alpha}$ against the function $\mathbb I_S\otimes \mathbb I_S\otimes \mathbb I_S$, where $\mathbb I_S$  is the function which is identically $1$  on $S$. Let  

\[ I(\boldsymbol \lambda) = K(\boldsymbol \alpha) = \int_S\int_S\int_S k_{\alpha_1}(x_2,x_3)k_{\alpha_2}(x_3,x_1)k_{\alpha_3}(x_1,x_2)d\sigma(x_1) d\sigma(x_2)d\sigma(x_3)\]
where $\boldsymbol \alpha$ and $\boldsymbol \lambda$ are related by the relations \eqref{param}.
\begin{proposition} Let $\boldsymbol \lambda = (\lambda_1,\lambda_2,\lambda_3)$ in $\mathbb C^3$ and assume that the conditions \eqref{mero3} are satisfied. Then
\begin{equation}\label{ival}
\begin{split}
 I(\boldsymbol\lambda)= \left({\frac{\sqrt \pi}{2}}\right)^{3(n-1)}
 2^{\lambda_1+\lambda_2+\lambda_3}\dots \hskip 3cm \\
 \dots\frac
 {\Gamma(\frac{\lambda_1+\lambda_2+\lambda_3+\rho}{2})\Gamma(\frac{-\lambda_1+\lambda_2+\lambda_3+\rho}{2})\Gamma(\frac{\lambda_1-\lambda_2+\lambda_3+\rho}{2})\Gamma(\frac{\lambda_1+\lambda_2-\lambda_3+\rho}{2})}
 {\Gamma(\rho+\lambda_1)\Gamma(\rho+\lambda_2)\Gamma(\rho+\lambda_3)} 
 \end{split}
\end{equation}
\end{proposition}

\noindent
{\bf Remark} Both sides of the formula are meromorphic functions on $\mathbb C^3
$, and they are equal where defined. Notice that  $I(\boldsymbol \lambda)= K(\boldsymbol \alpha)$ has simple poles \emph{exactly} as prescribed by Theorem \ref{mero}.

This result allows to strengthen the previous theorem. Define

\[\widetilde {\mathcal K}_{\boldsymbol \alpha} =\frac{\mathcal K_{\boldsymbol \alpha}}{ \Gamma(\frac{\lambda_1+\lambda_2+\lambda_3+\rho}{2})\Gamma(\frac{-\lambda_1+\lambda_2+\lambda_3+\rho}{2})\Gamma(\frac{\lambda_1-\lambda_2+\lambda_3+\rho}{2})\Gamma(\frac{\lambda_1+\lambda_2-\lambda_3+\rho}{2})}\ .
\]
\begin{theorem} The distribution-valued function $\boldsymbol \alpha\mapsto  \widetilde {\mathcal K}_{\boldsymbol \alpha}$ extends as an entire holomorphic function on $\mathbb C^3$. The trilinear functional $\tilde {\mathcal T}_{\boldsymbol \lambda}$ defined by
\[\widetilde {\mathcal T}_{\boldsymbol \lambda}(f_1, f_2, f_3)= \widetilde {\mathcal K}_{\boldsymbol \alpha}(f_1\otimes f_2\otimes f_3)
\]
on $\mathcal C^\infty(S)\times \mathcal C^\infty(S)\times \mathcal C^\infty(S))$ is invariant with respect to $(\pi_{\lambda_1}, \pi_{\lambda_2}, \pi_{\lambda_3})$, where $(\lambda_1,\lambda_2,\lambda_3)$ are related to $(\alpha_1,\alpha_2,\alpha_3)$ by the relations \eqref{param}. The trilinear form  $\widetilde {\mathcal T}_{\boldsymbol \lambda}$ is not identically $0$ provided the two following  conditions are not simultaneously realized

$\bullet \ \exists j , 1\leq j\leq 3, \quad  \lambda_j\in -\rho-\mathbb N$

$\bullet$ {(at least) one of the  conditions \eqref{mero3} is satisfied .}
\end{theorem}

\begin{proof} The function $\boldsymbol \alpha \longmapsto \widetilde {\mathcal K}_{\boldsymbol \alpha}$ extends holomorphically near any regular point of the hyperplanes defined by conditions \eqref{poles1} and \eqref{poles2}, so is holmomorphic outside of the set where ar least two hyperplanes of poles meet. But this set is of codimension $2$ and hence  $\boldsymbol \alpha \longmapsto \widetilde {\mathcal K}_{\boldsymbol \alpha}$ extends  as a holomorphic function to all of $\mathbb C^3$. If none of the conditions \eqref{mero3} is satisfied, then  $\widetilde {\mathcal K}_{\boldsymbol \alpha}$ is a multiple ($\neq 0$) of $ {\mathcal K}_{\boldsymbol \alpha}$ which is different from $0$ on $\Omega$. If $\lambda_j\notin -\rho+\mathbb N$ for $j=1,2,3$, then $I(\boldsymbol \lambda)\neq 0$ and hence $ {\mathcal K}_{\boldsymbol \alpha}$ is not identically $0$. 
\end{proof}

\section {Uniqueness of the invariant trilinear form}

\subsection{Induced representations and line bundles }
For this part, it is useful to realize the representation $\pi_\lambda$ as acting on smooth sections of a line bundle over $S$. 

Recall that the stabilizer of $\bf 1^+$ in $G$ is the parabolic subgroup  $P=MAN$. The left invariant Haar measure on $P$ is
\begin{equation}
\int_P f(p)dp = \int_M\int_A\int_N f(man)dm da dn\ .
\end{equation}
For $\lambda\in \mathbb C$, denot by $\chi_\lambda$ the character of $P$ defined by
\[ \chi_\lambda(m a_t n) = e^{t\lambda}\ .\]
The modular function of $P$ is given by $\delta_P(man) = e^{-2\rho(\log a)}= \chi_{-2\rho}(p)$ (see e.g. \cite{w} Lemma 5.5.1.1), so that, for any $q$ in $P$
\begin{equation}
\int_P f(pq)\,dp = \chi_{2\rho}(q) \int_Pf(p)\,dp
\end{equation}

Let $\mathcal E_\lambda$ be the space of functions $f$ in ${\mathcal C}^\infty (G)$ which, for all $g$ in $G$, $p$ in $P$ satisfy
\begin{equation}
f(gp) = \chi_{-(\rho+\lambda)}(p)f(g)
\end{equation}
Then $G$ acts on $\mathcal E_\lambda$ by
\begin{equation}
\Pi_\lambda(g)f(\gamma) = f(g^{-1}\gamma)
\end{equation}
To any function $f$ in $\mathcal E_\lambda$, associate the function $f^{\widetilde \ }$ defined on $S$ by the formula
\begin{equation}\label{rest}
 f^{\widetilde \ }(s) = f(k)
\end{equation}
where $k$ is any element in $K$ satisfying $k{\bf 1^+} = s$. As $f(km) = f(m)$ for any $m$ in $M$, the right handside of \eqref{rest} does not depend on $m$, but merely on $k{\bf 1^+} = s$.

Now, let $g$ be in $G$, let $s$ be in $S$, and  let $k$  be in $K$ such that $k{\bf 1^+} = s\in S$. Let
\[ g^{-1} k = k(g^{-1} k)\, a{(g^{-1} k)}\, n
\]
be the Iwasawa decomposition of $g^{-1} k$. Then
\[ k(g^{-1}k) {\bf 1^+} = (g^{-1} k) ({\bf 1^+})=g^{-1}(s)
\]
and
\[\kappa(g^{-1},s)= \kappa(g^{-1}k, {\bf 1^+})= \chi_{-1}({a(g^{-1} k)})\ .
\]

Hence,
\[
\big(\Pi_\lambda(g)f\big)^{\widetilde \ }(s) = f(g^{-1} k) = f(k(g^{-1} k)\, a(g^{-1} k)\, n) = \kappa(g^{-1},s)^{\rho+\lambda} \, f^{\widetilde \ }(g^{-1}(s))
\]
so that $f\mapsto  f^{\ \widetilde { }}$ is an intertwining operator for $\Pi_\lambda$ and $\pi_\lambda$.

Let $\chi$ be a character of $A$. Denote by $\mathbb C_\chi$ the representation of $P$ on $\mathbb C$ given by
\[man.z=\chi(a) z\ .
\]
Form the line bundle $L=L_\chi=G\times_P \mathbb C_\chi$ over $S$, and let $\mathcal L=\mathcal L_\chi^\infty$ be the space of $\mathcal C^\infty$ sections of $L$. Then $G$ acts naturally on $\mathcal L$ by left translations. As sections of $L$ can be identified with functions on $G$ transforming by $\chi^{-1}$ under the right action of $P$, the representation of $G$ on $\mathcal L$ is equivalent to $\pi_\lambda$ if $\chi=\chi_{\rho+\lambda}$.

Now take three characters $\chi_1=\chi_{\rho+\lambda_1},\chi_2=\chi_{\rho+\lambda_2},\chi_3=\chi_{\rho+\lambda_3}$ of $A$, set $\boldsymbol \lambda = (\lambda_1,\lambda_2,\lambda_3)$   and form the "exterior" product bundle
$L_{\boldsymbol \lambda} = L_{\chi_1}\boxtimes L_{\chi_2}\boxtimes L_{\chi_3}$ as a line bundle over $S\times S\times S$. Let   $\mathcal L_{\boldsymbol \lambda}$ be the space of $\mathcal C^\infty$ sections of this bundle. Let the group $G$ acts naturally on this space by diagonal action. Then a trilinear invariant functional on $\mathcal C^\infty (S)\times \mathcal C^\infty (S)\times \mathcal C^\infty (S)$ for $\pi_{\lambda_1}, \pi_{\lambda_2},\pi_{\lambda_3}$ corresponds to an invariant linear form on $\mathcal L_{\boldsymbol \lambda}$.

In turn, such an invariant linear functional on $\mathcal L_{\boldsymbol \lambda}$ can be viewed as an invariant distribution density for the dual bundle  $ L^*_{\boldsymbol \lambda}$ (see Appendix).

The main tool to study these invariant distributions is \emph {Bruhat's theory}, which is presented in the Appendix. We use heavily the description of orbits of $G$ in $S\times S\times S$ (cf Proposition \ref{orbit}).

\begin{theorem}\label{uniq} Let $\boldsymbol \lambda = (\lambda_1,\lambda_2,\lambda_3)$ be three complex numbers, let $\boldsymbol \alpha = (\alpha_1,\alpha_2,\alpha_3)$ be defined as in \eqref{param} and assume that they satisfy the following \emph{generic} conditions 

$i)$ $\alpha_j \notin -\rho-2\mathbb N$ for $j=1,2,3$

$ii)$ $\alpha_1+\alpha_2+\alpha_3\notin -\rho -2\mathbb N$.

Then any trilinear invariant form for $\pi_{\lambda_1}, \pi_{\lambda_2}, \pi_{\lambda_3}$ is proportional to the form $\mathcal T_{\boldsymbol \lambda}$.
\end{theorem}

\begin{proof} 
Denote by $T= T_{\boldsymbol \lambda}$ the  distribution density on $S\times S\times S$ for the bundle $L_{\boldsymbol \lambda}^*$ corresponding to an invariant trilinear form for $\pi_{\lambda_1}, \pi_{\lambda_2}, \pi_{\lambda_3}$.

\subsection*{Step 1 : contribution of $\mathcal O_0$}

We use freely of the notation presented in the appendix. Consider the restriction $T'$ of $T$ to the open orbit $\mathcal O_0$. Then $\mathcal O _0$ is a homogeneous space under $G$, the stabilizer $H$ of the base point $({\bf 1}^+, {\bf 1}^-, e_2)$ is compact. There is an invariant measure on $G/H$,  the group $H$ acts trivially on the fiber of $L_{\boldsymbol \lambda}^* $, so there exists exactly one (up to constant) invariant distribution given by a smooth density. But we already know that $K_{\boldsymbol\alpha}$ has the right transformation property. Hence on $\mathcal O_0$, after multiplication by a constant, we may assume that $T' $ coincides with the restriction of $\mathcal K_{\boldsymbol \alpha}$ to $\mathcal O_0$. But the assumptions on $\boldsymbol 
\lambda$ guarantee that $\mathcal K_{\boldsymbol \alpha}$ can be extended (by analytic continuation) as an invariant distribution on $S\times S\times S$. Hence, in order to prove that $T$ is a multiple of $\mathcal K_{\boldsymbol\alpha}$,  (i.e. to prove the uniqueness statement) we need only to prove that an invariant distribution which vanishes on $\mathcal O_0$ is identically $0$. In other words, we may (and hence do) assume that \[Supp(T)\subset \cup_{1\leq j \leq 4}\mathcal O_j.\] and proceed further to prove that $T=0$.

Let us mention that the argument given here should be modified for the case of the circle (i.e. as $n=2)$, because  there are two orbits for $SO_0(1,2)$ in $S\times S\times S$. To restore uniqueness, one can also consider the full group $O(1,2)$ instead. The rest of the proof is unchanged.

\subsection*{Step 2 : contribution of $\mathcal O_1,\mathcal O_2,\mathcal O_3$}
We now show that \[Supp(T)\cap \mathcal O_3 =\emptyset\ .\] 
Observe that for $1\leq j\leq 3$
\[\overline{\mathcal O_j} = \mathcal O_j \cup \mathcal O_4
\]
so that $(\mathcal O_1\cup \mathcal O_2\cup \mathcal O_4)$ is a closed subset of $S\times S\times S$.  Let \[X=S\times S\times S\setminus (\mathcal O_1\cup \mathcal O_2\cup \mathcal O_4)\ .\] 
Then $X$ is an open submanifold of $S\times S\times S$, acted by $G$, and $\mathcal O_3$ is a closed orbit of $G$ in $X$. Let $T''$ be the restriction of $T$ to $X$, so that $Supp(T'')\subset \mathcal O_3$. Now apply Bruhat's theory.

The normal space at the base point $(\bf 1^+, \bf 1^+,\bf 1^-)$ is identified (via the Riemannian metric on $S\times S\times S$) with
\[ N_o = \{ (\xi, -\xi, 0), \xi\in T_{\bf 1}S\}\simeq \mathbb R^{n-1}
\]
The stabilizer of of the base point is $H=AM$, and there is an invariant measure on $\mathcal O_3$, so that $\chi_0 \equiv 1$. The group $M$ acts on $N_0$ by its natural action. As $M$ modules, $N_0$ and its dual are equivalent, so that the spaces $S^k(N_0)$ and $\mathcal P_k(N_0)$  are equivalent $M$-modules. The space  of $M$-invariants in $\mathcal P_k(N_0)$  is $0$ if $k$ is odd (and one-dimensional, generated by $\vert \xi\vert^k$ if $k$ is even, but we won't need this fact). On the other hand, an element $a_t$ of A acts on $N_0$ by dilation by $e^{-t}$, so that it acts on $ S^k(N_0)$ by multiplication by $e^{-kt}$. The element $a_t$ acts on the fiber $L_0$ of $L_{\chi_1,\chi_2,\chi_3}$ at $(\bf 1^+, \bf 1^+,\bf 1^-)$ by $e^{(\rho+\lambda_1+\rho+\lambda_2-\rho-\lambda_3)t}$.  Hence the element $a_t$ acts on $S^k(N_0)\otimes L_0^*$ by multiplication by \[e^{(-k-\rho-(\lambda_1+\lambda_2-\lambda_3))t}.\]

The assumptions on $\boldsymbol \lambda$ (more precisely $\alpha_3 = \lambda_1+\lambda_2-\lambda_3 \notin -\rho-2\mathbb N$) guarantee that
\[  \big(\mathcal S_k(N_0)\otimes L_0^*\otimes \mathbb C_{\chi_0^{-1}}\big)^H
=0
\]
for any $k\in \mathbb N$. Hence, by Corollary A\ref{cor}, there is no non-trivial invariant distribution supported in $\mathcal O_3$.

 Repeating the argument for $\mathcal O_2$ and $\mathcal O_1$, we may (and hence do) assume now that \[Supp(T)\subset \mathcal O_4.\]
 
 \subsection*{Third step : contribution of $\mathcal O_4$}

 Here we take $X=S\times S\times S$, as $\mathcal O_4$ is closed. The stabilizer of the base-point $(\bf 1^+,\bf 1^+,\bf 1^+)$ is $P=MAN$. The character $\chi_0$ is given by
 \[ \chi_0(ma_tn)= e^{-2\rho t}\ .\]
 The normal space at $(\bf 1^+,\bf 1^+,\bf 1^+)$ can be identified with \[N_0 = \{(\xi_1,\xi_2,\xi_3), \xi_j \in T_{\bf 1}S, \xi_1+\xi_2+\xi_3=0\} \simeq \mathbb R^{n-1}\oplus \mathbb R^{n-1}\ .
 \]
 The group $M$ acts on $N_0$ by its natural action on each factor. Again, as $M$ modules, the space    $S^k(N_0)$  is isomorphic to $\mathcal P_k(N_0)$. The algebra of  $SO(n-1)$ invariant polynomials on $\mathbb R^{n-1}\oplus \mathbb R^{n-1}$ is generated (as an algebra) by $\vert \xi\vert ^2, \vert \eta\vert^2, \langle \xi,\eta\rangle$.  Hence $S^k(N_0)^M = \{0\}$ if $k$ is odd. On the other hand, an element $a_t$ of $A$ acts on $N_0$ by multiplication by $e^{-t}$, hence on $S_k(N_0)$ by multiplication by $e^{-kt}$. It acts on the fiber $L_0$ by $e^{(\rho+\lambda_1+\rho+ \lambda_2+\rho+ \lambda_3 )t}$. Hence $a_t$ acts on $S^k(N_0)\otimes L_0^*\otimes \mathbb C_{\chi_o^{-1}}$ by
\[e^{(-k-\rho -(\lambda_1 \lambda_2+ \lambda_3 ))t}\ .\]
The assumptions on $\boldsymbol \lambda$ (namely $\lambda_1+\lambda_2+\lambda_3\notin -\rho-2\mathbb N$) guarantee that 
\[ \big(\mathcal S_k(N_0)\otimes L_0^*\otimes \mathbb C_{\chi_0^{-1}}\big)^H
=\{0\}\ 
\]
for any $k\in \mathbb N$.  Hence, by Corollary A\ref{cor}, there is no non-trivial invariant distribution supported in $\mathcal O_4$. The uniqueness statement follows.
\end{proof} 

\section{Another construction of a trilinear invariant form}

In this section we present a different construction of the invariant trilinear form $\mathcal T_{\boldsymbol \lambda}$. Trilinear forms are connected with  tensor products of representations. Roughly speaking, a trilinear form on $\mathcal H_1\times \mathcal H_2\times \mathcal H_3$ can be realized as an invariant bilinear pairing between $\mathcal H_1\otimes \mathcal H_2$ and $\mathcal H_3$. Although this point of view breaks the (sort of) symmetry between the three factors, it produces interesting relations with questions about tensor products of representations (cf \cite{mol}, which was the main source of inspiration for this section). Our presentation of the construction is formal and we don't work out the estimates and analytical aspects of the construction, which would follow along similar lines as in previous sections.

\subsection{ The space $S^2_\top$ as a symmetric space}

Recall that $S^2_\top= \{(x,y)\in S\times S , x\neq y\}$ is the orbit of $(\bold 1^+, \bold 1^-)$ under $G$. The stabilizer in $G$ of $(\bold 1^+, \bold 1^-)$ is the subgroup $H= MA$.

Let $J=\begin{pmatrix}-1&0&0&\dots&0\\0&-1&0&\dots &0\\0&0&1& &0\\\vdots&\vdots& &\ddots&\vdots\\0&0&0&\dots&1\end{pmatrix}$. Then $J^t=J^{-1} = J$. The map \[g\longmapsto JgJ\] is an involutive automorphism of $G$. The set of fixed points of this involution is the subgroup $\widetilde H = H\sqcup H^-$,  
\[H^- = \Bigg\{h= \begin{pmatrix}\cosh t&-\sinh t&0\\\sinh t & -\cosh t &0\\ 0&0&k \end{pmatrix},\quad t\in \mathbb R, k\in O(n-1), \det k = -1\Bigg\} \ .
\]
So $H$ is the neutral component of the fixed points of an involutive automorphism of $G$. In other words, $S^2_\top$ can be realized as the \emph{symmetric space}  $G/H$  via
\[ G/H \ni g \longmapsto \big(g({\bf 1}^+), g({\bf 1}^-)\big)\ .\]

There exists a unique (up to a positive real number) $G$-invariant measure on $S^2_\top$, namely
\begin{equation}
\int_{G/H} f(x) d\mu(x) = \int_S\int_S f(s,t) \frac{ds\ dt}{\vert s-t\vert^{2(n-1)}}\ .
\end{equation}
Another description of $G/H$ is as follows. Let $\mathcal X$ be the set of all $2$-dimensional oriented subspaces $\Pi$ in $\mathbb R^{1,n}$ such that the restriction of $q$ to $\Pi$ is of signature $(1,1)$. The space $\mathcal X$ is an open set in the Grassmannian $G(2,n+1)$ of oriented $2$-dimensional subspaces in $\mathbb R^{1,n}$. The group $G$ operates transitively on $\mathcal X$. To the couple $(s,t)$ in $S^2_\top$, associate the $2$-dimensional space $\Pi(s,t) = \mathbb R \widetilde s\oplus \mathbb R \widetilde t$, with the orientation for which $(\widetilde s,\widetilde t)$ is a direct basis.

\begin{proposition}\label{Xmodel}
The mapping $\Pi$ is a diffeomorphism from $S_\top^2$ onto $\mathcal X$.
\end{proposition} 
\begin{proof}
Let $s,t$ be in $S^2_\top$. Then $\Pi(s,t)$ cannot be totally isotropic for $q$, because the maximally isotropic subspaces are of dimension $1$. As $\Pi(s,t)$ contains two independant isotropic vectors ($\widetilde s$ and $\widetilde t$), the signature of $q$ on $\Pi(s,t)$ has to be $(1,1)$. Hence $\Pi(s,t) $ belongs to $\mathcal X$. On the other hand, let $\Pi$ be in $\mathcal X$. Then the isotropic cone in $\Pi$ decomposes as the union of two distinct isotropic lines, say $D\cup \Delta$. The line $D$ (resp. $\Delta$) intersects the hyperplane $\{x_0 = 1\}$ at a point $\widetilde s$ (resp. $\widetilde t$) with $s$ (resp. $t$) in $S$. Permuting $D$ and $\Delta$ if needed, we may assume that the basis $(\widetilde s,\widetilde t)$ corresponds to the orientation of $\Pi$. Then clearly $\Pi(s,t) = \Pi$. Thus $\Pi$ is a $1-1$ correspondance. The fact that it is a diffeomorphism is standard and left to the reader.
\end{proof}

\subsection{An equivariant realization of $\pi_\sigma\otimes \pi_\tau$}

The group $H$ is connected, isomorphic to $\mathbb R\times SO(n-1)$. For $\zeta$ any complex number, the function $\nu_\zeta$ on $H$  defined by
\[ \nu_\zeta \begin{pmatrix}\cosh t&\sinh t&0\\\sinh t & \cosh t &0\\ 0&0&k \end{pmatrix}= e^{t\zeta}\]
$(t\in \mathbb R, k\in M)$ is a character of $H$. 

Form the line bundle $L_\zeta = G \times _H \mathbb C_\zeta$ over $G/H$.  A section of $L_\zeta$ can be viewed as a $\mathcal C^\infty$ function $F : G\longrightarrow \mathbb C$ which satisfies
\begin{equation}\label{zetainv}
F(gh) = \nu_\zeta(h)^{-1} F(g) \ .
\end{equation}
Denote by $\mathcal L_\zeta$ the space of smooth sections of $L_\zeta$ over $G/H$. The natural action of  $G$ acts  on  $\mathcal L_\zeta$ by left translation is denoted by $I_\zeta$ :
\[\big(I_{\zeta} (\gamma) F\big)(g) = F(\gamma^{-1}g)\ .
\]

When $\zeta$ is pure imaginary, the character $\nu_\zeta$ is unitary. If $F$ is in $\mathcal L_\zeta$, then $\vert F(gh) \vert = \vert F(g)\vert$ for any $h$ in $H$. Hence the expression
\begin{equation}\label{L2section}
\Vert F\Vert ^2 = \int_{G/H}\vert F(g)\vert^2 d\mu(gH)\ , 
\end{equation}
is well-defined (maybe $+\infty$) and is finite if (say) $F$ has compact support modulo $H$. The representation $I_\zeta$ extends as a unitary representation for this inner product.
 
Let $\sigma, \tau$ be two complex numbers, and consider the representations $\pi_\sigma$ and $ \pi_\tau$. As $C^\infty(S)$ is a Fr\' echet space, the projective and inductive topologies on the tensor product $C(S)\otimes C(S)$ coincide, and the (completed) tensor product is realized as $C^\infty(S\times S)$. Hence the tensor product $\pi_\sigma \otimes \pi_\tau$ is naturally realized on $\mathcal C^\infty(S\times S)$. Explicitly,
\begin{equation}
\pi_\sigma\otimes \pi_\tau (g) f(x_1,x_2) = \kappa(g^{-1}, x_1)^{\rho+\sigma} \kappa(g^{-1}, x_2)^{\rho+\tau}f(g^{-1}(x_1), g^{-1}(x_2))
\end{equation}
for $g\in G, f\in \mathcal C^\infty(S\times S), x_1,x_2\in S$.

For $f$ in $\mathcal C^\infty(S\times S)$, let $P_{\sigma,\tau}f$  be the function on $G$ defined by
\begin{equation}
\big(P_{\sigma,\tau} f\big)(g) =\kappa(g,{\mathbf 1^+})^{\rho+\sigma}\kappa(g,{\mathbf 1^-})^{\rho+\tau}f\big(g({\mathbf 1^+}), g({\mathbf 1^-})\big)
\end{equation}
\begin{proposition} $P_{\sigma, \tau} f$ satisfies the relation
\begin{equation}\label{invH}
\big(P_{\sigma, \tau} f\big) (gh) = \nu_{\sigma-\tau}(h)^{-1} \big(P_{\sigma, \tau} f\big)(g), \end{equation}
for $g$ in $G$ and $h$ in $H$. 
\end{proposition}

\begin{proof} Recall that the elements of $H$ fix both ${\mathbf 1^+}$ and ${\mathbf 1^-}$. Moreover, 
\begin{equation}\label{Hkappa}
\kappa(a_t, {\mathbf 1^\pm}) = e^{\mp t}, \quad t\in \mathbb R\ .
\end{equation}
Now \eqref{invH} follows  \eqref{kappa}. 
\end{proof}
The map $P_{\sigma, \tau}$ will be regarded as a map from $C^\infty(S\times S)$ into $\mathcal L_{\sigma-\tau}$.

Let $\sigma$ and $\tau$ be pure imaginary. Observe that
\[\vert g({\bf 1}^+)-g({\bf 1}^-)\vert =2\, \kappa(g,{\bf 1}^+)^{\frac{1}{2}}\kappa(g,{\bf 1}^-)^{\frac{1}{2}}\ ,
\]
so that
\[\Vert P_{\sigma, \tau} f \Vert ^2 = 2^{-2(n-1)}\int_{S\times S} \vert s-t\vert^{2(n-1)} \vert f(s,t) \vert^2 \frac{ ds\ dt}{\vert s-t\vert^{2(n-1)}}= 2^{-2(n-1)}\Vert f\Vert^2
\]
and $P_{\sigma, \tau}$ extends as an isometry (up to a constant) from $L^2(S\times S)$ onto the space of square-integrable elements of $\mathcal L_{\sigma-\tau}$ for the inner product associated to \eqref{L2section}. 

\begin{proposition}  $P_{\sigma, \tau}$ intertwines the representation $\pi_\sigma\otimes \pi_\tau$ and the representation $I_{\sigma-\tau}$, i.e. for any $g$ in $G$
\begin{equation}\label{intertw}
P_{\sigma, \tau} \circ (\pi_\sigma\otimes \pi_\tau) (g) = I_{\sigma-\tau}(g)\circ P_{\sigma, \tau}
\end{equation}
\end{proposition}
\begin{proof}
Let $f$ be in $\mathcal C^\infty(S\times S)$. For $\gamma$ in $G$, 
\[ P_{\sigma,\tau}\big((\pi_\sigma\otimes \pi_\tau)(\gamma)f)(g) =\]
\[{\kappa(g,{\mathbf 1}^+)}^{\rho+\sigma}{\kappa(g,{\mathbf 1}^-)}^{\rho+\tau}\kappa(\gamma^ {-1}, g({\mathbf 1^+}))^{\rho+\sigma} \kappa(\gamma^ {-1}, g({\mathbf 1^-}))^{\rho+\tau}\dots \]\[\dots f\big(\gamma^{-1}(g({\mathbf 1^+}), g({\mathbf 1^-}))\big)\ .\]
On the other hand, 
\[ I_{\sigma-\tau}(\gamma) P_{\sigma, \tau}f(g) =\]
\[   \kappa(\gamma^{-1}g, {\mathbf 1^+})^{\rho+\sigma}\kappa(\gamma^{-1}g, {\mathbf 1^-})^{\rho+\tau}
 f\big(\gamma^{-1}( g({\mathbf 1^+}), g^{-1}({\mathbf 1^-}))\big)\ .
\]
The two expressions are easily seen to be equal by using  \eqref{kappa}.
\end{proof}

\subsection{Construction of a $(H,\nu_\zeta)$ covariant function for $\pi_\lambda$}
The next step consists in finding in the representation space for $\pi_\lambda$ an element $\Theta_{\lambda,\zeta}$ which transforms under the action of $H$ by the character $\nu_\zeta$. In general, it will be a distribution on $S$ (a commun fact in harmonic analysis on semisimple symmetric space (cf \cite{vdb}). We use a geometric approach through the realization of $G/H$ as $\mathcal X$ (see Proposition \ref{Xmodel}).

Let $\Pi$ be an element of $\mathcal X$ and $s$ be in $S$. Observe that $\Pi^\perp$ is a an $(n-1)$-dimensional space, complementary to $\Pi$, and the restriction of $q$ to $\Pi^\perp$ is negative-definite. Define 
\begin{equation}
\Psi(\Pi, s) = 2\big( -q(\proj_{\Pi^\perp} \widetilde s)\big)^{\frac{1}{2}}\ .
\end{equation}

\begin{proposition}Let $(s_1,s_2)$ be in $S^2_\top$. Then, for any $s_3$ in $S$, 
\begin{equation}
\Psi(\Pi(s_1,s_2), s_3) =2 \frac{\vert s_1-s_3\vert \vert s_2-s_3\vert}{\vert s_1-s_2\vert}\ .
\end{equation}
\end{proposition}

\begin{proof} Let $\sigma_3=\proj_{\Pi^\perp} \widetilde s_3$. Then there exist real numbers $\alpha_1$ and $\alpha_2$ such that $\sigma_3 = \widetilde s_3-\alpha_1 \widetilde s_1-\alpha_2 \widetilde s_2$. The numbers $\alpha_1$ and $\alpha_2$ are determined by the conditions
\[ [\sigma_3, \widetilde s_1] = [\sigma_3, \widetilde s_2] = 0\ .\]
Hence,
\[  1-s_3.s_2-\alpha_1 (1-s_1.s_2) = [\widetilde s_3, \widetilde s_2] -\alpha_1 [\widetilde s_1, \widetilde s_2] = [\sigma_3+\alpha_2 \widetilde s_2, \widetilde s_2] = 0
\]
so that 
\[\alpha_1 = \frac{1-s_2.s_3}{1-s_1.s_2} = \Big(\frac{\vert s_2-s_3\vert}{\vert s_1-s_2\vert}\Big)^2
\]
and similarly
\[ \alpha_2 = \frac{1-s_3.s_1}{1-s_1.s_2}= \Big(\frac{\vert s_1-s_3\vert}{\vert s_1-s_2\vert}\Big)^2\ .
\]
Now
\[ -q(\sigma_3) = -[\widetilde s_3-\alpha_1\widetilde s_1-\alpha_2 \widetilde s_2, \widetilde s_3-\alpha_1\widetilde s_1-\alpha_2 \widetilde s_2]\]
\[ =2 \alpha_1[\widetilde s_1,\widetilde s_3]+2 \alpha_2 [\widetilde s_2,\widetilde s_3]-2\alpha_1\alpha_2[\widetilde s_1,\widetilde s_2])
\]
\[ =\Big( \frac{\vert s_1-s_3\vert \vert s_2-s_3\vert}{\vert s_1-s_2\vert}\Big)^2\ .
\] \end{proof}
\begin{proposition}Let $g$ be in $G$. Then, for any $\Pi$ in $\mathcal X$ and $s$ in $S$,
\begin{equation}
\Psi(g(\Pi), g(s)) = \kappa(g,s) \Psi(\pi,s)
\end{equation}

\end{proposition}

\begin{proof}As $g$ preserves the form $q$ and hence the orthogonality relative to $q$, for any $\Pi$ in $\mathcal X$
 \[ g\circ \proj_{\Pi^\perp} = \proj_{g(\Pi)^\perp}\circ \,g\ .\]
 Hence
\[ q(proj_{g(\Pi)^\perp}g(s)) = (g\,\widetilde s)_0^{-2}q(proj_{g(\Pi)^\perp} g\,\widetilde s) = (g\,\widetilde s)_0)^{-2}q(\proj_{\Pi^\perp} \,s)\]
so that 
\[\Psi(g(\Pi),g(s))= (g\,\widetilde s)_0^{-1} \Psi(\Pi,s) = \kappa(g,s) \Psi(\pi,s)
\]
\end{proof}

Let $\Pi_0 = \Pi({\mathbf 1^+},{\mathbf 1^-})= \big\{\widetilde x=(x_0,x_1,0,\dots,0), x_0,x_1\in \mathbb R\}$, and let 
\begin{equation}
\Psi_0(s) = \Psi(\Pi_0, s) = \vert \mathbf 1^+-s\vert\vert \mathbf 1^- -s\vert =2 (x_2^2+\dots +x_n^2)^{\frac{1}{2}}\ .
\end{equation}
As a consequence of the previous proposition, the function $\Psi_0$ has a nice transformation law under the action of $H$.
\begin{lemma} Let $h= \begin{pmatrix}\cosh t&\sinh t&0\\\sinh t & \cosh t &0\\ 0&0&k \end{pmatrix}$, where $t$ is in $\mathbb R$, and $k$ is in $SO(n-1)$. For any $s$ in $S$
\begin{equation}
\Psi_0(h(s)) = \kappa(h,s) \Psi_0(s)\ .
\end{equation} 

\end{lemma}

Let $\lambda$ and $\zeta$ be two complex numbers. For $s$ in $S$ define

\begin{equation}\label{Hfix}
\Theta_{\lambda, \zeta} (s) = \Psi_0(s)^{-\rho-\lambda}\ \Big\vert \frac { {\mathbf 1^+}-s}{{\mathbf 1^-}-s}\Big\vert^\zeta = \vert \mathbf 1^+-s\vert^{-\rho-\lambda+\zeta}\vert \mathbf 1^--s\vert^{-\rho-\lambda-\zeta} \ .
\end{equation}

\begin{proposition}
 The function $\Theta_{\lambda, \zeta}$ satisfies 
\begin{equation}\label{theta}
\pi_\lambda(h) \Theta_{\lambda, \zeta} = \nu_\zeta(h) \Theta_{\lambda, \zeta}
\end{equation}
for all $h$ in $H$.
\end{proposition}
\begin{proof} Let $h$ be in $H$. Then
\[ \pi_\lambda(h)\Theta_{\lambda,\zeta} (s) = \kappa(h^{-1},s)^{\rho+\lambda}\,\Theta_{\lambda, \zeta}(h^{-1}(s))\]
\[=\kappa(h^{-1}, s)^{\rho+\lambda}\, \Psi_0(h^{-1}(s))^{-\rho-\lambda}\, \Big\vert\frac{{\mathbf 1^+}-h^{-1}(s)}{{\mathbf 1^-}-h^{-1}(s)}\Big\vert^\zeta\]
\[ = \kappa(h^{-1},s)^{\rho+\lambda}\, \kappa(h^{-1},s)^{-\rho-\lambda}\,\Psi_0(s)^{-\rho-\lambda}\, \Big\vert\frac{\kappa(h^{-1}, {\mathbf1^+})}{\kappa(h^{-1}, {\mathbf1^-})}\,\Big\vert^{\frac{\zeta}{2}} \,\Big\vert\frac{{\mathbf 1^+}-s}{{\mathbf 1^-}-s}\Big\vert^\zeta\]
\[ = \Psi_0(s)^{-\rho-\lambda}\, (e^{2t})^{\frac{\zeta}{2}}  \,\Big\vert\frac{{\mathbf 1^+}-s}{{\mathbf 1^-}-s}\Big\vert^\zeta = \nu_\zeta(h) \,\Theta_{\lambda,\zeta} (s)\ .\]
by using \eqref{Hkappa}.
\end{proof}

\subsection{ The duality between $I_\zeta$ and $\pi_\lambda$}
Now define the corresponding \emph{Fourier transform} $\mathcal F_{ \lambda, \zeta}$ : for $F$ a smooth section of $L_\zeta$ with compact support modulo $H$, define $\mathcal F_{\zeta,\lambda}F$ by
\begin{equation}\mathcal F_{ \lambda,\zeta} F (s) = \int_{G/H} F(g)\, \pi_\lambda(g)\Theta_{\lambda, \zeta}(s)\, d\mu(gH)\ .
\end{equation}
Observe that, thanks to \eqref{theta}  the integrand is a function on $G/H$, so that the integral makes sense, the result being in general a distribution on $S$.
\begin{proposition} For any $g$ in $G$
\begin{equation}\label{Fourier}
\mathcal F_{ \lambda,\zeta}\, I_\zeta (g) = \pi_\lambda(g) \mathcal F_{ \lambda, \zeta,}\ .
\end{equation}

\end{proposition}

\begin{proof}Set $\Theta = \Theta_{\lambda, \zeta}$, $\mathcal F = \mathcal F_{\lambda, \zeta}$, let $F$ be in $\mathcal L_\zeta$ with compact support modulo $H$ and let $\gamma$ be in $G$. Then
\[ \mathcal F \circ I_\zeta(\gamma)F (s) =  =\int_{G/H} F(\gamma^{-1}g)\ \kappa(g^{-1},s)^{\rho+\lambda}\ \Theta(g^{-1}(s)) d\mu(gH)\ .
\]
Set $\gamma^{-1}g = l$ and use the invariance of the measure $d\mu$ to obtain
\[\mathcal F \circ I_\zeta(\gamma)F (s)=\int_{G/H} F(l) \kappa(l^{-1} \gamma^{-1},s)^{\rho+\lambda}\Theta(l^{-1}(\gamma^{-1}(s)) d\mu(lH)
\]
\[ =\kappa(\gamma^{-1}, s)^{\rho+\lambda}\int_{G/H} F(l) \kappa\big(l^{-1}, \gamma^{-1}(s)\big)^{\rho+\lambda}\Theta\big(l^{-1}( \gamma^{-1}(s)\big) d\mu(lH)
\]
\[ =\kappa(\gamma^{-1}, s)^{\rho+\lambda}\int_{G/H} F(l)\pi_\lambda(l) \Theta(\gamma^{-1}(s))d\mu(lH)
\]
\[ =\kappa(\gamma^{-1},s)^{\rho+\lambda} \mathcal F F (\gamma^{-1}(s))
 = \big(\pi_\lambda(\gamma) \mathcal F F\big) (s)\ .\]
\end{proof}
\subsection{Application to trilinear forms}
\begin{theorem} Let $\lambda_1,\lambda_2,\lambda_3$ be three  complex numbers. For $f_1,f_2,f_3$ three functions in $\mathcal C^\infty(S)$, let $T(f_1,f_2,f_3)$ be defined by
\begin{equation}\label{tri2}
T(f_1,f_2,f_3) = \Big( \mathcal F_{ -\lambda_3,\,\lambda_1-\lambda_2} P_{\lambda_1,\lambda_2} (f_1\otimes f_2), f_3\Big)\ .
\end{equation}
Then, for any $g\in G$, 
\[ T(\pi_{\lambda_1}(g) f_1, \pi_{\lambda_2}(g) f_2,\pi_{\lambda_3}(g) f_3)= T(f_1,f_2,f_3) \ ,
\]
whenever the right hand side is well defined.
\end{theorem}

\begin{proof} Let $g$ be in $G$ and let $f_1,f_2,f_3$ be three functions in $\mathcal C^\infty(S)$. Then
\[ T(\pi_{\lambda_1}(g) f_1, \pi_{\lambda_2}(g) f_2,\pi_{\lambda_3}(g) f_3)=  \Big( \mathcal F_{ -\lambda_3, \lambda_2-\lambda_1,} P_{\lambda_1,\lambda_2} (\pi_{\lambda_1} (g)f_1\otimes \pi_{\lambda_2}(g)f_2), \pi_{\lambda_3}(g)f_3\Big)
\]
\[ = \Big( \mathcal F_{ -\lambda_3, \lambda_1-\lambda_2} I_{\lambda_1-\lambda_2}(g) P_{\lambda_1,\lambda_2}(f_1\otimes f_2), \pi_{\lambda_3}(g) f_3\Big)
\]
(use \eqref{intertw})
\[= \Big( \pi_{-\lambda_3}(g) \big(\mathcal F_{-\lambda_3, \lambda_1-\lambda_2}P_{\lambda_1,\lambda_2}(f_1\otimes f_2)\big), \pi_{\lambda_3} (g) f_3\Big)
\]
(use \eqref{Fourier})
\[ = T(f_1,f_2,f_3)\ .\]
(use \eqref{dua}). 
\end{proof}
Making explicit the right handside  of \eqref{tri2} shows that it coincides (up to a constant) with the former expression of $\mathcal T_{\boldsymbol \lambda}$ where $\boldsymbol \lambda = (\lambda_1,\lambda_2,\lambda_3)$ (see Theorem \ref{Kalpha}).

\section*{Appendix : Invariant distributions supported by a submanifold}

Here is a presentation of the main results in Bruhat's theory \cite{bru}, but written in terms of vector bundles and  distribution  densities (in the sense of \cite{ho} ch. VI). We sketch the main steps of the proof, following \cite{w}.

\subsection*{A.1 Distribution densities for a vector bundle over a manifold}

First recall the composition of a distribution with a diffeomorphism. Let $X_1$ and $X_2$ be two open sets of $\mathbb R^N$ and let $\Phi : X_1\rightarrow X_2$ be a $\mathcal C^\infty$ diffeomorphism. Then there is a unique \emph{continous} linear map $\Phi^*: \mathcal D'(X_2)\rightarrow \mathcal D'(X_1)$ which extends the composition of functions, i.e. such that $\Phi^* f = f\circ \Phi$ for $f\in \mathcal C(X_2)$.

A \emph {distribution density} on a manifold $X$ is by definition a continuous linear form on $\mathcal C^\infty_c(X)$. Let $u$ be a distribution density on $X$. Let $(X_\kappa, \kappa)$ be a local chart, i.e. $\kappa$ is a diffeomorphism of an open set $X_\kappa$ of $X$ onto an open set $\widetilde X_\kappa$ of $\mathbb R^N$. Then  the formula
\begin{equation}\label{density2}
u_\kappa(\varphi) =  u(\varphi\circ \kappa)
\end{equation}
for $\varphi\in \mathcal C^\infty_c(\widetilde X_\kappa)$ defines a distribution $u_\kappa$ on $\widetilde X_\kappa$, called the \emph{local expression} of $u$ in the chart $(X_\kappa, \kappa)$ . Further, let $(X_\kappa, \kappa)$ and $(X_{\kappa'}, \kappa')$ be two overlapping charts, and let $u_\kappa$ and $u_{\kappa'}$ the corresponding local expressions of $u$. Let 
\[\Phi : \kappa'(X_\kappa\cap X_{\kappa'})\rightarrow   \kappa(X_\kappa\cap X_{\kappa'})
\]
be the change of coordinates (equal to $\kappa\circ \kappa'^{-1}$). Then
\begin{equation}\label{density1}
u_{\kappa'}= \vert \det d\Phi\vert\, \Phi^*u_\kappa \quad {\rm in}\quad\kappa'(X_\kappa\cap X_{\kappa'})\ .
\end{equation}

Conversely, suppose we have an atlas $\mathcal F$ of charts $(X_\kappa, \kappa)$
covering the manifold $M$ and suppose that for each $\kappa$ we are given a distribution $u_\kappa \in \mathcal D'(\widetilde X_\kappa)$. Assume further that for any two overlapping charts $(X_\kappa,\kappa)$ and $(X_{\kappa'}, \kappa')$, the  condition \eqref{density1} is satisfied.
Then the system $\big((u_\kappa), \kappa\in \mathcal F\big)$ defines a unique distribution density $u$ on $X$  such that, for  $\kappa$ in $\mathcal F$ and $\varphi \in \mathcal C_c^\infty(\widetilde X_\kappa)$ condition \eqref{density2} is satisfied.

The space of distribution densities on $X$ is denoted by $\mathcal D'(X)$. A \emph{smooth density} is a density the local expressions of which are $\mathcal C^\infty$ functions. The smooth densities are $\mathcal C^\infty$ sections of a line bundle called the \emph{density bundle} $\Omega(X)$. It is very similar to the bundle of differential forms of maximal degree on $X$, in the sense that  their transition functions just differ by an absolute value.

This definition can be extended to the case of vector bundles. Let $\pi: L\rightarrow X$ be a $\mathcal C^\infty$ vector bundle over $M$, with model fiber  $E_0$. Let $(X_i)$ be a family of open subsets of $X$ such that over each $X_i$ the bundle can be trivialized.  Let $\Psi_i : \pi^{-1} (X_i) \rightarrow X_i\times E_0$ and $\Psi_j : \pi^{-1} (X_j)\rightarrow X_j\times  E_0$ be two trivializations of the bundle over two overlapping subsets $X_i$ and $X_j$. Then the map $ \Psi_i\circ \Psi_j^{-1}$ over $ (X_i\cap X_j)\times L_0\rightarrow (X_i\cap X_j)\times L_0$ is of the form
\[  (x,  v)\mapsto (x, g_{ij} (x) \, v)\]
where  $g_{ij}(x)$ is in $GL(L_0)$ and the map (\emph{transition functions of the bundle})
 \[ g_{ij}:\ X_i\cap X_j\in x \mapsto g_{ij} (x)\in GL(E_0)\] is $\mathcal C^\infty$.
A distribution density for the bundle $L$ is a system $(u_i)$ of distribution densities on $X_i$ with values in $L_0$ such that 
\[ u_i=g_{ij} u_j \quad {\rm in}\ X_i \cap X_j\ .\]
Denote by $\mathcal D'(X,L)$ the space of distribution densities for the bundle $L$.

Let $\mathcal L_c^\infty$ be the space of $C^\infty$ sections with compact support of the bundle $L$. Then the dual of $\mathcal L_c^\infty$ is identified with the space $\mathcal D'(X,L^*)$. The smooth elements in the dual  (those given locally by integration against a smooth function) are the $\mathcal C^\infty$  sections of the bundle $ \Omega(X)\otimes L^*$.
 \subsection* {A.2 Invariant distribution : the case  of a homogeneous vector bundle}

Let $G$ be a Lie group acting transitively on a manifold $X$. Let $o$ be a base-point in $X$, and $H=G^o$ be its stabilizer in $G$, so that $X\simeq G/H$.  An element $h$ of $H$ acts on $X$ and fixes $o$, so that by differentiation, it acts on its tangent plane by (say) $\tau_0(h)$. The tangent space $T_0X$ can be identified with $\mathfrak g / \mathfrak h$. The element $h$ acts on $\mathfrak g$ by the adjoint action $\Ad_{\mathfrak g} (h)$, and preserves the subalgebra $\mathfrak h$ on which it acts by $\Ad_{\mathfrak h} h$. Hence its acts on $\mathfrak g/\mathfrak h$, and this action coincides with $\tau_0(h)$. This action satisfies
\begin{equation}
\det \tau_0(h) = \frac{\det(\Ad_{\mathfrak g}h)}{\det(\Ad_{\mathfrak h}h)}
\end{equation}

The \emph{modular function} $\delta_G$ of a Lie group  $G$ is defined by
\[\delta_G(g) = \vert \det \Ad(g^{-1})\vert\ .
\]
so that 
\begin{equation}\label{modular}
\vert \det \tau_0(h)\vert \ : = \chi_0(h) = \frac{\delta_H(h)}{\delta_G(h)}\quad .
\end{equation}

A \emph{homogeneous vector bundle} $L$ over $X$ is a vector bundle $L$ together with  an action of the group $G$ on $L$ by bundle isomorphisms.  If $g$ is in $G$ and $x$ in $X$,  then $g$ maps the fiber $L_x$ into $L_{g(x)}$ by a linear isomorphism. In particular, $H$ acts on $L_o$ by a representation (say) $\tau$. Conversely, given a representation $\tau$ of $H$ in a vector space $E$, then one constructs the bundle
$G\times_{\tau} E$ as $G\times E/\sim$, where $\sim$ is the equivalence relation defined by the right action of $H$ on $G\times E$
\[(g,v) \sim(gh^{-1}, \tau(h)v) \quad {\rm for\ some\ }  h\in H\ ,
\]
$g$ in $G$ and $v\in E$. Any homogeneous vector bundle over $X$ is of that sort, in the sense that the bundle $L$ is isomorphic to $G\times_{\tau} L_0$ (see \cite{wal}).

A section $s:X\longrightarrow L$ can be realized as a $L_0$-valued function $f_s$ on   $G$ which satisfies
\begin{equation}\label{fibersection}
f_s(gh) =\tau^{-1}(h) f_s(g)
\end{equation}
and, conversely, such a function $f$ gives raise to a section of $L$. Let $\mathcal L_c^\infty$ be the space of smooth sections of $L$ with compact support. The space $\mathcal L_c^\infty$ is $G$-equivariantly isomorphic to  $\mathcal C^\infty_c(G, H,\tau)$, the space of $\mathcal C^\infty$ functions on $G$ which satisfy \eqref{fibersection} and have compact support modulo $H$. The group $G$ acts by left translations on $\mathcal L_c^\infty$, and this action is equivariant with the left action of $G$ on $\mathcal C^\infty_c(G, H,\tau)$.

The tangent bundle $TX$ of $X$ is an example of such a homogeneous bundle. The action of $H$ on the fiber $T_0X$ is $\tau_0$. Another important homogenous bundle is the bundle $\Omega(X)$ of densities over $X$.  It is a line bundle, corresponding to the character  of $H$ given by $\quad \vert\det \big((\tau(h)^{-1})^t\big)\vert = \chi_o(h)^{-1}\quad $  (cf \eqref{modular}), so that, in this context, we denote the fiber at $o$ of the bundle $\Omega(X)$ by $\mathbb C_{\chi_0^{-1}}$. 

\begin{theorem*}\label{invdist} There exists a non trivial invariant continous linear forms on $\mathcal L_c^\infty$ if and only if there exists a non trivial linear form $\xi$ on $E$ such that, for all $h \in H$
\begin{equation}\label{invv}
\tau(h^{-1})^t \xi = \chi_0(h) \xi \ . 
\end{equation}
More precisely,
\begin{equation}
\big({{\mathcal L_c^\infty}^*\big)}^G= \mathcal D'(X, L^*)^G \simeq (L_0^*\otimes \mathbb C_{\chi_0^{-1}})^H.
\end{equation}
\end{theorem*}

\begin{proof} (Sketch of) An element of the dual of $\mathcal L_c^\infty$ is a distribution density for the bundle $L^*$.  If it is invariant by $G$, then the corresponding distribution density turns out to be smooth. Hence we are looking for a $G$-invariant smooth section of the bundle $L^*\otimes \Omega(X)$. But this is equivalent to an $H$ invariant element in the fiber at $o$.
\end{proof}

\subsection*{A.3 Invariant distribution supported in a submanifold}

Let $X$ be a manifold and $G$ a Lie group acting on $X$. Let $Q$ be an orbit of $G$ in $X$ and assume that $Q$ is closed. Let $L$ be a homogeneous vector bundle over $X$. Let  $N$ be the normal bundle of $Q$ (i.e. the quotient bundle $TX_{\vert Q}/TQ)$. Fix a base-point $o$ in $Q$, let $G_o=H$ be the stabilizer of $o$. Then $H$ acts on the tangent space $T_oX$, preserving the subspace $T_oQ$, and hence acts on the normal space $N_o = T_oX/T_oQ$ at $o$.

Let $T$ be a distribution on $Q$. Then the map 
\[\mathcal C^\infty_c(X)\ni \varphi \longmapsto  (T,\varphi_{\vert Q})
\]
defines a distribution on $X$, which we still denote by $T$.

Let $\mathcal L_c^\infty$ be the space of smooth sections of $L$ with compact support.  Let $T$ be a continuous linear form on $\mathcal L^\infty_c$  supported in $Q$.  Choose a local coordinate system on $X$
\[(u_1,\dots, u_s, v_1,\dots, v_r)\] such that $v_j=0, 1\leq j\leq r$ are local equations for $Q$. The $(u_i)_{1\leq i\leq s}$ form a coordinate system of $Q$ near $o$. On the other hand let $N$ be the normal bundle of $Q$. The family $(\frac{\partial}{\partial v_j})_{ 1\leq j\leq r}$ gives a local trivialization  of the normal bundle $N$. Choose a local trivialization of the bundle $L$, and denote by $e^*_1,\dots, e^*_l$ the corresponding coordinates on the fiber.

By Schwarz's local structure theorem for distributions supported in a vector subspace, there exists an integer $k\in \mathbb N$ (the local \emph{transversal order} of the distribution), and for each $j, 1\leq j\leq l$ and each multidiindex ${\boldsymbol \alpha}=(\alpha_1,\dots, \alpha_r)$ with $ \vert {\boldsymbol \alpha}\vert = \alpha_1+\dots+\alpha_r\leq k$  uniquely determined distributions $T_{\boldsymbol \alpha}^j$ on (some open subset of) $\mathbb R^s\subset \mathbb R^s\times \mathbb R^r$ such that 
\[ T= \sum_{1\leq j \leq l}\sum_{\vert {\boldsymbol \alpha}\vert \leq k} (-1)^{\vert {\boldsymbol \alpha} \vert}D^{\boldsymbol \alpha} T_{\boldsymbol \alpha}^j\  e_j^*
\]
Let us consider the "top terms" subcollection $(T_{\boldsymbol \alpha}^j, \vert\boldsymbol \alpha\vert = k, 1\leq j\leq l)$. This collection can be interpreted as the local expression of a distribution density for the bundle $S^k(N)\otimes L_{\vert Q}^*$ on $Q$. This is obtained by checking the way the collection transforms under

$\bullet$ change of local trivialization of the bundle $L$ (change the $e_j^*$'s)

$\bullet$ change of local trivialization of the normal bundle (change the $v_j$'s)

$\bullet$ change of the local coordinate system on $Q$ (change the $u_j$'s).

Denote by $\sigma^{(k)}(T)$ the section of the bundle $S^k(N)\otimes L_{\vert Q}^*$ associated to $T$.

Denote by $\mathcal D'_k(Q,L)$ the space of continuous linear forms on $\mathcal L_c^\infty$ which are supported on $Q$ and of transversal order $\leq k$ on any local chart of $X$ .

\begin{theorem*}\label{invhom}The map $T\longmapsto \sigma^{(k)}(T)$ is a linear map from $\mathcal D'_k(Q,L)$ in $\mathcal D'\big(Q, S^k(N)\otimes L_{\vert Q}^*\big)$, with kernel $\mathcal D'_{k-1}(Q,L)$. Moreover, if 
$\Phi$ is a bundle diffeomorphism of the bundle $L$, which maps $Q$ into $Q$, then
\[\sigma^{(k)}({\Phi^*T}) =\Phi^*\sigma^{(k)}(T)\ ,
\]
where $\Phi^*$ denotes the action naturally  induced by $\Phi$ on $\mathcal D'(Q,L)$ or of the restriction of $\Phi_{\vert L_{\vert Q}}$ on  $\mathcal D'(Q, S^k(N)\otimes L_{\vert Q}^*)$.

\end{theorem*}

Assume now that $T$ is invariant by $G$. Obersve that invariance insures that the local transversal degree of $T$ is the same for all charts of $Q$.

Denote by $\mathcal D'_k(Q,L)^G$ the space of $G$-invariant elements of $\mathcal D'_k(Q,L)$. The conjonction of Theorem A\ref{invdist} and Theorem A\ref{invhom} gives some estimate of its dimension. 
\begin{theorem*}\label{invsub}
For any $k\in \mathbb N$,
\[\dim\big( \mathcal D'_k(Q,L)^G/\mathcal D'_{k-1}(Q,L)^G\big) \leq \dim \big(\mathcal S_k(N_0)\otimes L_0^*\otimes \mathbb C_{\chi_0^{-1}}\big)^H\ .
\]
\end{theorem*}

The theorem is mostly used in through the following corollary.

\begin{corollary*}\label{cor} Assume that for any $k\in \mathbb N$
\[
\big(\mathcal S_k(N_0)\otimes L_0^*\otimes \mathbb C_{\chi_0^{-1}}\big)^H = \{0\}\ .
\]
Then there exists no non trivial continuous $G$-invariant linear form on $\mathcal L_c^\infty$ such that $Supp(T)\subset Q$.
\end{corollary*}

\bigskip
\footnotesize{ \noindent Addresses\\ (JLC) Institut \'Elie Cartan, Universit\'e Henri Poincar\'e (Nancy 1),
54506 Vandoeuvre-l\`es-Nancy, France.\\
(B\O ) Matematisk Institut, Byg.\,430, Ny Munkegade, 8000 Aarhus C,
Denmark.
\medskip

\noindent \texttt{{jlclerc@iecn.u-nancy.fr, 
 orsted@imf.au.dk,
}}

\end{document}